\documentclass[smallextended]{svjour3}  

\smartqed  
\journalname{JOTA}

\newtheorem{assume}{Assumption}
\usepackage{graphicx}
\usepackage{subfigure}
\usepackage{amsmath,amssymb}
\usepackage{mathrsfs}
\usepackage[T1]{fontenc}
\usepackage{color}
\usepackage{cases}
\usepackage{diagbox}
\usepackage{multirow}
\usepackage{tcolorbox}
\usepackage[figuresleft]{rotating}
\usepackage{booktabs}
\usepackage[ruled,linesnumbered]{algorithm2e}

\usepackage[colorlinks,linkcolor=blue,anchorcolor=blue,citecolor=blue,urlcolor=blue]{hyperref}
\usepackage{epstopdf}

\newcommand{\R}{\mathbb{R}}
\newcommand{\E}{\mathbb{E}}
\newcommand{\N}{\mathbb{N}}
\newcommand{\calI}{\mathcal{I}}
\newcommand{\x}{\boldsymbol{x}}
\newcommand{\y}{\boldsymbol{y}}
\newcommand{\z}{\boldsymbol{z}}
\newcommand{\calC}{\mathcal{C}}
\newcommand{\calP}{\mathcal{P}}
\newcommand{\whV}{\widehat{\bs{V}}}
\newcommand{\balpha}{\boldsymbol{\alpha}}
\newcommand{\bbH}{\mathbb{H}}
\newcommand{\calB}{\mathcal{B}}
\newcommand{\calWB}{\widetilde{\mathcal{B}}}

\newcommand{\orcid}[1]{\href{https://orcid.org/#1}{\includegraphics[scale=1]{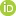}}}
\newcommand{\bs}[1]{\boldsymbol{#1}}
\DeclareMathOperator*{\argmin}{argmin}
\DeclareMathOperator*{\prox}{prox}
\DeclareMathOperator*{\proj}{P}

\def\fundingname{Funding}

\newenvironment{myblock}[1]{\par\addvspace{17pt}\small\rmfamily
\trivlist\item[\hskip\labelsep
{\bfseries#1}]}{}

\begin{document} 
\title{A Boosted-DCA with Power-Sum-DC Decomposition for Linearly Constrained Polynomial Programs\thanks{Accepted by the Journal of Optimization Theory and Applications February 15, 2024. \\Communicated by Alper Yildirim.}}

\titlerunning{A Boosted-DCA with Power-Sum-DC Decomposition for Polynomial Programs}

\author{Hu Zhang \and Yi-Shuai Niu \orcid{0000-0002-9993-3681}}
\authorrunning{Hu Zhang \and Yi-Shuai Niu} 

\institute{
		Hu Zhang  \at
		School of Mathematical Sciences, Shanghai Jiao Tong University, China \\
		\and
		Yi-Shuai Niu \at
		Beijing Institute of Mathematical Sciences and Applications, Beijing, China\\
		niuyishuai@bimsa.cn; niuyishuai82@hotmail.com
}
\date{Received: date / Accepted: date}
\maketitle

\begin{abstract}
This paper proposes a novel Difference-of-Convex (DC) decomposition for polynomials using a power-sum representation, achieved by solving a sparse linear system. We introduce the Boosted DCA with Exact Line Search (BDCA$_\text{e}$) for addressing linearly constrained polynomial programs within the DC framework. Notably, we demonstrate that the exact line search equates to determining the roots of a univariate polynomial in an interval, with coefficients being computed explicitly based on the power-sum DC decompositions. The subsequential convergence of BDCA$_\text{e}$ to critical points is proven, and its convergence rate under the Kurdyka-{\L}ojasiewicz property is established. To efficiently tackle the convex subproblems, we integrate the Fast Dual Proximal Gradient (FDPG) method by exploiting the separable block structure of the power-sum DC decompositions. We validate our approach through numerical experiments on the Mean-Variance-Skewness-Kurtosis (MVSK) portfolio optimization model and box-constrained polynomial optimization problems. Comparative analysis of BDCA$_\text{e}$ against DCA, BDCA with Armijo line search, UDCA, and UBDCA with projective DC decomposition, alongside standard nonlinear optimization solvers \texttt{FMINCON} and \texttt{FILTERSD}, substantiates the efficiency of our proposed approach.
\end{abstract}

\section{Introduction}
We are interested in solving the polynomial optimization problem with linear constraints:
\begin{equation*}
	\label{eq:POPL}
 	\min_{\x\in\calP}f(\x),\tag{POPL}
\end{equation*}
where $f$ is a multivariate polynomial, and $\calP$ represents a nonempty convex polyhedral set (not necessarily bounded), defined by
\begin{equation*}
	\calP:=\{\x\in\R^n:
	\langle \bs{a}_i,\x\rangle\le b_i, \langle \bs{p}_j,\x\rangle =q_j, i=1,\ldots,m,j=1,\ldots,r\}.
\end{equation*}
Here, $\bs{a}_i,\bs{p}_j\in\R^n$ and $b_i,q_j\in\R$. Throughout the paper, 
the polynomial function $f$ is assumed to be lower bounded over $\calP$.

This problem arises in many applications including, but not limited to, higher-order moment portfolio selection \cite{Maringer2009GlobalOO,dinh2016dc,Niu2011,niu2019higher}, eigenvalue complementarity problem \cite{le2012dc,niu2013efficient,niu2015solving,niu2019improved}, tensor complementarity problem \cite{song2016properties}, Euclidean distance
matrix completion problem \cite{bakonyi1995euclidian}, copositivity of matrix \cite{bras2016copositivity,dur2013testing,you2022refined}, Boolean polynomial program \cite{niu2019discrete}, and bilinear matrix inequalities \cite{boyd1994linear,niu2014dc}, among other. Each of these applications underscores the essential and widespread utility of addressing the stated polynomial optimization problem with linear constraints, thereby highlighting the imperative nature of devising efficient and robust solutions for it.

The \eqref{eq:POPL} is indeed a special class of the difference-of-convex (DC) program \cite{hartman1959functions,hiriart1985generalized,le2018dc}, defined as below:  
\begin{equation}\label{eq:P-dc}
	\min_{\x\in\mathcal{\calP}}\{f(\x):=g(\x)-h(\x)\}, \tag{P$_\text{dc}$}
\end{equation}
where $g$ and $h$ are convex polynomials. There are some related works on DC decompositions for polynomials. The DC decomposition for quadratic functions has been studied in \cite{hoai2000efficient,bomze2004undominated,le1997solving} based on the eigenvalue decomposition of real symmetric matrices. However, in the case of general multivariate polynomials, this technique is no longer available and the DC decomposition is more difficult to obtain (especially in dense polynomials). Niu et al. \cite{Niu2011} 
proposed a DC decomposition for general polynomials in the form of $\frac{\sigma}{2}\|\x\|^2-(\frac{\sigma}{2}\|\x\|^2-f)$ where $\sigma>0$ is obtained by estimating an upper bound of the spectral radius of the Hessian matrix of $f$ over a compact convex set. This technique has been successfully applied to several real-world applications such as the higher-order moment portfolio optimization \cite{Niu2011} and the eigenvalue complementarity problems \cite{niu2013efficient,niu2015solving,niu2019improved}. Ahmadi et al. \cite{ahmadi2018dc} explored DC decompositions for polynomials via algebraic relaxation techniques. They showed in \cite{ahmadi2012convex} that convexity is different from sos-convexity, and characterized the discrepancy between convexity and sos-convexity in \cite{ahmadi2013complete} for polynomials. Several DC decompositions are proposed in \cite{ahmadi2018dc} by solving linear, second-order cone, or semidefinite programs. Meanwhile, Niu defined in \cite{niu2018difference} the so-called Difference-of-Convex-Sums-of-Squares (DC-SOS) decomposition for general polynomials and developed several practical DC-SOS decomposition techniques. Some of these decomposition techniques are also parallelizable. The DC decomposition technique proposed in this paper is indeed a special DC-SOS decomposition in form of power-sums of linear forms, namely the power-sum DC (PSDC) decomposition. We will show that this decomposition can be established by solving a sparse linear system, which is discussed in Section \ref{sec:dec_reform}.

Concerning the solution method for DC programming formulation \eqref{eq:P-dc}, the most popular method is DCA, introduced by Pham \cite{1988Duality} in 1985 as an extension of the subgradient method, and extensively developed by Le Thi and Pham since 1994 (see, e.g.,\cite{le1997solving,tao2005dc,le2018dc,1997Convex,tao1998dc} and the references therein). The main idea of DCA is to minimize a sequence of convex majorizations of the DC objective function $f$ by linearizing the second DC component $h$ at the current iterate. The general convergence theorem of DCA ensures that every limit point of the generated sequence $\{\x^k\}$ by DCA is a critical point of \eqref{eq:P-dc} (see, e.g., \cite{1997Convex,tao2005dc}). In recent years, some accelerated DCAs are established. Artacho et al. \cite{artacho2018accelerating} proposed a boosted DCA (BDCA) by incorporating DCA with an Armijo-type line search to potentially accelerate DCA under the smoothness assumption of both $g$ and $h$. This method is extended to the non-smooth case of $h$ component in \cite{aragon2018boosted}. Meanwhile, Niu et al. \cite{niu2019higher} further extended BDCA to the convex constrained DC program for both smooth and nonsmooth cases. The global convergence theorem is established under the {\L}ojasiewicz subgradient inequality. Besides, some other accelerations based on the heavy-ball method \cite{polyak1964some} and Nesterov’s extrapolation \cite{nesterov1983method} are also proposed. The inertial DCA (InDCA) \cite{de2019inertial} is established by de Oliveira et al. as a heavy-ball type \cite{polyak1964some} accelerated DCA. Le Thi et al. proposed in \cite{nhat2018accelerated} a Nesterov-type accelerated DCA (ADCA). Wen et al. \cite{wen2018proximal} proposed a proximal DCA algorithm with Nesterov-type extrapolation  \cite{gotoh2018dc}. In this paper, based on the BDCA framework for convex constrained DC program established in \cite{niu2019higher}, we introduce a variant of BDCA with exact line search for \eqref{eq:POPL}.

\emph{Contribution.} (i) We propose two special DC-SOS decompositions in the power-sum of linear forms (namely, the power-sum DC decomposition) for multivariate polynomials. These decompositions can be efficiently generated by solving sparse linear systems. (ii) We introduce BDCA with an exact line search (BDCA$_\text{e}$) for solving \eqref{eq:POPL}, wherein the line search entails computing roots of a univariate polynomial in an interval, and an upper bound for the initial step size of the line search is estimated. (iii) The convex subproblems required in BDCA$_\text{e}$ are tackled by employing the Fast Dual Proximal Gradient (FDPG) method \cite{beck2014fast}, which leverages the separable block structure of the power-sum DC decompositions. (iv) The convergence analysis of BDCA$_\text{e}$ is established akin to the methodology in \cite{niu2019higher}, and its convergence rate under the Kurdyka-{\L}ojasiewicz property is substantiated.

We delineate three major advantages of our proposed power-sum DC decompositions. Firstly, they furnish effective DC decompositions for polynomials. Secondly, the convex subproblem in BDCA$_\text{e}$ can be efficiently resolved in parallel due to the unique power-sum structure. Thirdly, the exact line search is simplified to the process of locating the roots of a univariate polynomial, with coefficients being computed explicitly based on the power-sum DC decompositions. These advantages are also regarded as major contributions of this paper.

The rest of this paper is organized as follows. In Section \ref{sec:pre}, we recall some notations and preliminaries required in the paper. In Section \ref{sec:dec_reform}, we establish two power-sum DC decompositions to get two DC programming formulations for \eqref{eq:POPL}. The corresponding BDCA$_\text{e}$ with the two DC formulations are proposed in Section \ref{sec:DCAe}. The convergence analysis and the convergence rate of BDCA$_\text{e}$ are proved in Section \ref{sec:con-analysis}. The FDPG method for efficiently solving the convex subproblem is discussed in Section \ref{sec:FDPG}. Some numerical experiments on the Mean-Variance-Skewness-Kurtosis portfolio optimization model and the polynomial optimization problem with box constraint are reported in Section \ref{sec:numerical-experiment}. Some concluding remarks and important questions are summarized in the final section.

\section{Notations and preliminaries}\label{sec:pre}
Throughout this paper, matrices and vectors are written in bold uppercase letters and lowercase letters (e.g., $\bs{X}$ and $\bs{x}$) respectively. We use $x_i$ to denote the $i$-th coordinate of the vector $\bs{x}$, and $\bs{X}^{\top}$ and $\bs{X}^{-1}$ to denote the transpose matrix and the inverse matrix of $\bs{X}$. Let $\R^n$ be the $n$-dimensional Euclidean space endowed with the classical inner product $\langle\x,\y\rangle:=\sum_{i=1}^n x_iy_i$ for $\x,\y\in \R^n$, the induced Euclidean norm $\|\x\|:=\sqrt{\langle\x,\x\rangle}$ and the $\ell_{\infty}$ norm $\|\x\|_{\infty}:=\max_{1\le i\le n}\lvert x_i\rvert$. For $\bs{X}\in\R^{n\times m}$, $\|\bs{X}\|$ denotes the spectral norm of $\bs{X}$ defined by $\|\bs{X}\|: = \sqrt{\lambda_{\max}(\bs{X}^{\top}\bs{X})}$, where $\lambda_{\max}(\bs{X}^{\top}\bs{X})$ is the largest eigenvalue of the matrix $\bs{X}^{\top}\bs{X}$. Let $\bs{e}\in\R^n$ be the vector of ones and $\N:=\{0,1,\ldots\}$ be the set of natural numbers. The gradient of a differentiable function $f:\R^n\rightarrow \R$ at $\x\in\R^n$ is denoted by $\nabla f(\x)\in\R^n$.

A function $f:\R^n\rightarrow \R$ is convex if 
$$f(\lambda\x+(1-\lambda)\y)\le\lambda f(\x)+(1-\lambda)f(\y)\text{ for all } \x,\y\in \R^n \text{ and } \lambda\in(0,1),$$
and $f$ is called $\rho$-strongly convex ($\rho>0$) if  $f-\frac{\rho}{2}\|\cdot\|^2$ is convex.

A function $\psi: \R^n\rightarrow \R^n$ is called Lipschitz continuous, if there is some constant $L>0$ such that
$$\|\psi(\x)-\psi(\y)\|\le L\|\x-\y\|\text{ for all }\x,\y\in\R^n,$$
furthermore, $\psi$ is said to be locally Lipschitz continuous if, for each $\x\in\R^n$, there exists a neighborhood $\mathcal{U}$ of $\x$ such that $\psi$ restricted to $\mathcal{U}$ is Lipschitz continuous.

For a proper closed function $f:\R^n\to (-\infty,\infty]$, the Fr\'echet
subdifferential of $f$ at $\x_0\in \text{dom}f:=\{\x\in\R^n:f(\x)<\infty\}$ is defined by
$$\partial^{F}f(\x_0)=\{\bs{u}\in\R^n: \liminf_{\bs{h}\rightarrow \bs{0}}\dfrac{f(\x_0+\bs{h})-f(\x_0)-\langle\bs{u},\bs{h}\rangle}{\|\bs{h}\|}\ge0\},$$
and if $\x_0\notin \text{dom}f$, we set $\partial^{F}f(\x_0)=\emptyset.$ A point $\x_0\in \R^n$ is called a Fr\'echet critical point of $f$, if $0\in\partial^{F}f(\x_0).$ The effective domain of $\partial^Ff$ is defined by 
$$\text{dom} (\partial^Ff):=\{\x\in\text{dom} f: \partial^{F}f(\x)\ne\emptyset\}.$$ Particularly, when $f$ is convex, $\partial^{F}f$ coincides with the classical subdifferential in convex analysis, defined by 
$$\partial f(\x_0)=\{\bs{u}\in\R^n: f(\x)-f(\x_0)\ge\langle\bs{u},\x-\x_0\rangle,\forall\x\in\R^n\}.$$

Let $\mathrm{\Gamma}_0(\R^n)$ denote by the set of all proper closed and convex functions from $\R^n$ to $(-\infty,\infty]$. For a function $f\in\mathrm{\Gamma}_0(\R^n)$, its conjugate function $f^*$ is defined by
$$f^*(\y):=\sup_{\x}\{\langle \y,\x \rangle - f(\x): \x\in \R^n\}, \forall \y\in \R^n.$$
The proximal mapping of $f$ at $\x\in \R^n$ is given by
$$\prox\nolimits_{f}(\x):=\argmin_{\bs{u}}\{f(\bs{u}) + \frac{1}{2}\|\x-\bs{u}\|^2 :  \bs{u}\in\R^n\}.$$  
We call that $f$ is `prox-friendly' if $\prox\nolimits_{f}(\x)$ is easy to compute. See a list of prox-friendly convex functions in \cite[Chapter 7]{beck2017first}. For a nonempty closed convex set $\mathcal{C}$ of $\R^n$, the indicator function of $\mathcal{C}$ is defined by 
\begin{equation*}
    \chi_{\mathcal{C}}(\x):=
    \begin{cases}
     0 & \text{if } \x\in\mathcal{C},\\
     \infty& \text{otherwise}.
    \end{cases}
\end{equation*}
Note that $\chi_{\calC}$ belongs to $\mathrm{\Gamma}_0(\R^n)$ and $\partial \chi_{\calC}(\x)=\mathcal{N}_{\calC}(\x)$ if $\x\in \calC$, where $\mathcal{N}_{\calC}(\x)$ denotes the normal cone of $\calC$ at $x$, defined by
$$\mathcal{N}_{\calC}(\x):=\{\bs{v}\in \R^n: \langle \bs{v}, \bs{y}-\x \rangle \leq 0, \forall \bs{y} \in \calC\},$$
and $\mathcal{N}_{\calC}(\x) = \emptyset$ if $\x\notin \calC$.

A function $f:\R^n\rightarrow [-\infty,\infty]$ is called difference-of-convex (DC) if $f=g-h$ with $g, h\in\mathrm{\Gamma}_0(\R^n)$, where $g$ and $h$ are called DC components of $f$. By introducing $\chi_{\calP}$ into the component $g$ of the DC program \eqref{eq:P-dc} as
$$\min_{\x}\{(g+\chi_{\calP})(\x) - h(\x): \x\in \R^n\},$$
a point $\x^*\in \calP$ is called a (DC) critical point of \eqref{eq:P-dc} if and only if
$$0\in \partial (g + \chi_{\calP})(\x^*) - \partial h(\x^*).$$ 
If $h$ is differentiable at $\x\in \R^n$, then one has
$$\partial^{F}f(\x)=\partial g(\x)-\nabla h(\x).$$

In particular, to the case where both $g$ and $h$ are polynomials and $\calP$ is a convex polyhedral set, then $\x^*\in \calP$ is a DC critical point of \eqref{eq:P-dc} if and only if $$\nabla h(x^*) - \nabla g(\x^*) \in \mathcal{N}_{\calP}(\x^*).$$

The cone of feasible directions of $\calP$ at $\bar{\y}\in\calP$ is defined as $$\mathcal{D}(\bar{\y}):=\{\bs{d}\in\R^n: \exists \delta>0 \text{ such that } \bar{\y}+t\bs{d}\in\calP, \forall t\in[0,\delta]\},$$
and the active set at $\bar{\y}$ is defined by 
$$\mathcal{A}(\bar{\y})=\{i\in\{1,\ldots,m\}:\langle\bs{a}_i,\bar{\y}\rangle=b_i\}.$$
Clearly, 
$$\mathcal{D}(\bar{\y})=\{\bs{d}\in\R^n:\langle\bs{a}_i,\bs{d}\rangle\le0,\langle \bs{p}_j,\bs{d}\rangle=0,i\in\mathcal{A}(\bar{\y}),j=1,\ldots,r\}.$$
Let $\mathbb{H}_d[\x]$ denote by the set of all $n$-variable and $d$-degree homogeneous polynomials (forms) with coefficients in $\mathbb{R}$, defined by
$$\bbH_d[\x]:=\{f=\sum\nolimits_{\balpha}c_{\balpha}{\x}^{\balpha}:\x\in\R^n, c_{\balpha}\in\R,\balpha\in\calI_{n,d}\}.$$ Here, $\mathcal{I}_{n,d}$ denotes the set encompassing all weak compositions of $d$ into $n$ parts, defined by $\calI_{n,d}=\{\balpha=(\alpha_1,\ldots,\alpha_n):\sum_{i=1}^n\alpha_i=d,\alpha_i\in\N\}$. The cardinality of this set is denoted by $s_{n,d}$, is $\binom{n+d-1}{d}$.Thus, the set of all real polynomials with degree up to $d$ in $n$ variables is defined by $\R_d[\x]:=\cup_{k=0}^d \bbH_k[\x]$. The power-product matrix associated with $\calI_{n,d}$ is given by  $$\whV(n,d):=\bigg(\tbinom{d}{\balpha^j}(\balpha^i)^{\balpha^j}\bigg)_{1\le i,j\le s_{n,d}},$$
where $\tbinom{d}{\balpha^j}:=\frac{d!}{\alpha^j_1!\cdots\alpha^j_n!}$ is the multinomial coefficient and $(\balpha^i)^{\balpha^j}:=\prod_{k=1}^n(\alpha_k^i)^{\alpha_k^j}$. The power-product matrix $\whV(n,d)$ is nonsingular for all positive integers $n$ and $d$, and is sparse if $n\gg d$. See \cite{niu2023power} for more properties on the power-product matrix.

A form $f(\x)\in\bbH_d[\x]$ is said to have a \emph{power-sum} representation, if there exist linear forms $L_i(\x)\in\bbH_1[\x]$ and scalars $\lambda_i\in\R$ ($i=1,\ldots,r$) such that  $$f(\x)=\sum_{i=1}^{r}\lambda_iL_i^d(\x), \forall \x\in \R^n.$$

A proper and closed function $f:\R^n\to (-\infty,\infty]$ is said to have the Kurdyka--\L ojasiewicz (KL) property at $\x^*\in \text{dom}(\partial^{F} f)$, if there exist some constant $\eta\in(0,\infty)$, a concave function $\varphi: [0,\eta]\rightarrow [0,\infty)$ and a neighborhood $\mathcal{U}$ of $\x^*$, such that 
	\begin{enumerate}
		\item [(i)] $\varphi$ belongs to the class $\mathscr{C}^1$ on $(0,\eta);$
		\item [(ii)] $\varphi(0)=0$ and $\varphi^{\prime}>0$ on $(0,\eta);$
		\item [(iii)]$\forall \x\in \mathcal{U}$ with $f(\x)-f(\x^*)\in (0,\eta)$, we have the KL inequality
		$$\varphi^{\prime}(f(\x)-f(\x^*))\text{dist}(0,\partial^{F}f(\x))\ge 1.$$
\end{enumerate}
The class of functions endowed with the KL property is very ample. For example, all semi-algebraic and real subanalytic functions adhere to the KL inequality \cite{bolte2007lojasiewicz,kurdyka1994wf,lojasiewicz1993geometrie}. Consequently, the sum of polynomial functions and indicator functions associated with polyhedral sets also exhibits the KL property.

\section{Power-Sum-DC decompositions of polynomials}\label{sec:dec_reform}
\subsection{Power-sum representation}\label{subsec:power-sum}
Numerous studies have established the existence of a power-sum representation for homogeneous polynomials (see e.g., \cite{biosca2001representation,reznick1992sum}). In the quadratic case, the power-sum representation is nothing but spectral decomposition. 
Reznick outlined a necessary and sufficient condition for constructing power-sum decompositions of binary forms in \cite[section 5]{reznick1992sum}, providing a cornerstone for the subsequent formulation of an efficient algorithm to decompose binary polynomials into power-sums of linear forms with the minimal term count, as expounded in \cite{comon1996decomposition}. In the more general case, Biosca proved in \cite{biosca2001representation} that the set of the $d$-th power of all linear forms $\{L^d: L\in \bbH_1[\x]\}$ is a generating set for $\bbH_d[\x]$.
Nevertheless, Biosca's proof relied on induction, leaving a general method for identifying a finite generating set for any polynomial as a still-unresolved challenge. In our work, we unveil a finite generating set applicable to any polynomial $\varphi(\x)\in \bbH_d[\x]$. Furthermore, we elucidate that constructing a power-sum decomposition for $\varphi(\x)$ becomes a tractable task through the solution of a sparse linear system.
Lemma \ref{lem:psrep} summarizes the essential result.
\begin{lemma}\label{lem:psrep}
Let $\varphi(\x)$ be a form in $\bbH_d[\x]$, expressed as $\varphi(\x)=\sum_{\balpha\in \calI_{n,d}}c_{\balpha}{\x}^{\balpha},$ where $\bs{c}:=\big(c_{\balpha^1},\ldots,c_{\balpha^{s_{n,d}}}\big)^{\top}$. 
Then, a power-sum representation of $\varphi(x)$ is given by
	\begin{equation}\label{eq:psrep}
		\varphi(\x)=\sum_{\balpha\in \calI_{n,d}}\lambda_{\balpha}\langle\balpha, \x\rangle^d,
	\end{equation}
where $\bs{\lambda}=\big(\lambda_{\balpha^1},\ldots,\lambda_{\balpha^{s_{n,d}}}\big)^{\top}$ is the unique solution of the linear system
\begin{equation}\label{eq:linear-system}
	\whV(n,d)\bs{\lambda}=\bs{c}.    
\end{equation}
\end{lemma}
\begin{proof}
Let $\x\in \R^n$, the multinomial equation reads as
\begin{equation*}
    \bigg(\sum_{i=1}^{n} x_i\bigg)^d = \sum_{\balpha\in \calI_{n,d}} {d \choose \balpha} \x^{\balpha}.
\end{equation*}
Subsequently, utilizing the relationship
$$\langle\balpha, \x\rangle^d  = \sum_{\boldsymbol{\beta}\in \calI_{n,d}} {d \choose \boldsymbol{\beta}} \balpha^{\boldsymbol{\beta}} \x^{\boldsymbol{\beta}}, \forall \balpha\in \calI_{n,d},$$
we deduce that
\begin{equation}\label{eq:iden}
\calWB=\whV(n,d)\calB,
\end{equation}
where $\calB=(\x^{\balpha^1},\x^{\balpha^2},\ldots,\x^{\balpha^{s_{n,d}}})^{\top}$ and $\calWB=(\langle\balpha^1,\x\rangle^d,\langle\balpha^2,\x\rangle^d,\ldots,\langle\balpha^{s_{n,d}},\x\rangle^d)^{\top}$. Given the nonsingularity of $\whV(n,d)$ proved in \cite{niu2023power}, we can uniquely rewrite \eqref{eq:iden} as
$$\calB=\whV^{-1}(n,d)\calWB.$$ 
Hence $$\varphi(\x)=\bs{c}^{\top}\calB=\bs{c}^{\top} \whV^{-1}(n,d)\calWB=\bs{\lambda}^{\top}\calWB,$$
where $\bs{\lambda}=(\whV^{-1}(n,d))^{\top}\bs{c}$.
\qed\end{proof} 

\begin{remark}
As demonstrated above, a power-sum representation of forms can be obtained by solving the linear system as expressed in Equation \eqref{eq:linear-system}. It's important to note that the matrix $\whV(n,d)$ is asymmetric in general, and its size is $\tbinom{n+d-1}{d}^2$, which could be very large even if $n$ and $d$ are not too large.  For instance, when $d=4$ and $n=100$, the magnitude of $\whV(n,d)$ exceeds millions. However, owing to the sparsity of $\whV(n,d)$ particularly when $n\gg d$, solving Equation \eqref{eq:linear-system} becomes remarkably efficient. In most real-world applications of polynomial optimization, the number of variables $n$ significantly exceeds the degree $d$, hence the matrix $\whV(n,d)$ is sparse. Fig. \ref{fig:spa} provides some illustrative examples. More sparsity\footnote{The sparsity of $\bs{M}\in\R^{n\times m}$ is defined by the ratio of zero entries to the total number of entries in the matrix} properties of $\whV(n,d)$ is referred to \cite{niu2023power}.
\begin{figure}[ht!]
	\centering
	\includegraphics[width=10cm,height=6cm]{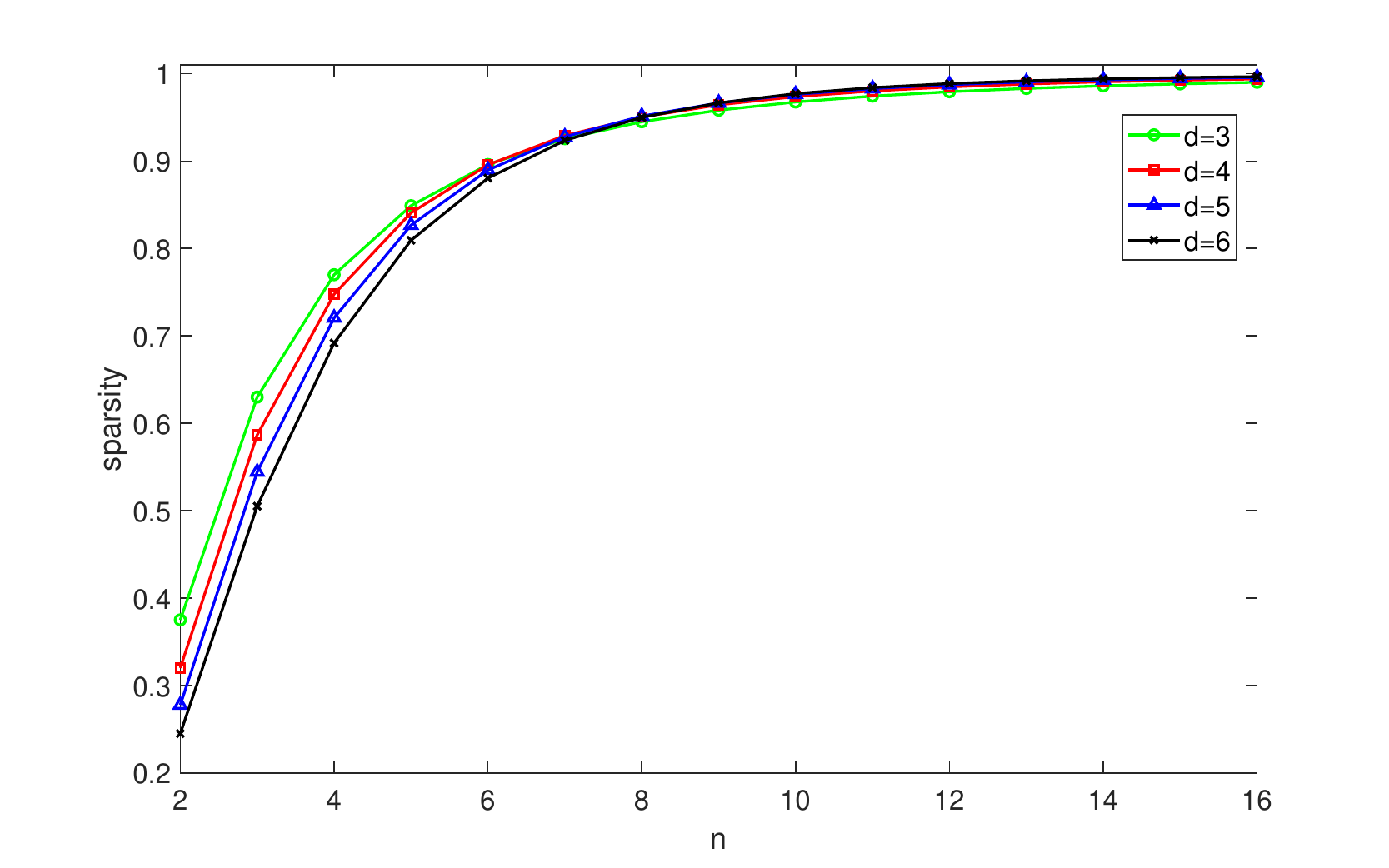}
	\caption{The sparsity of the matrix $\whV(n,d)$ changes with variable $n$ from $2$ to $16$ for some fixed degree $d\in\{3,4,5,6\}$. It is observed that, for $n>6$, the sparsity of $\whV(n,d)$ for all $d\in\{3,4,5,6\}$ is more than $90\%$, and the sparsity goes to $1$ as $n\rightarrow\infty$ for any fixed $d$.}
	\label{fig:spa}
\end{figure}
\end{remark}
\subsection{Power-sum difference-of-convex decompositions}\label{subsec:PSDC}
In this section, we establish two DC decompositions for any polynomial, namely the \emph{termwise power-sum-DC} (T-PSDC) decomposition and the  \emph{homogenizing-dehomogenizing power-sum-DC} (HD-PSDC) decomposition. 
\paragraph{\textbf{T-PSDC decomposition}}: We will consider two cases (the odd degree case and the even degree case) in the next proposition.
\begin{proposition}\label{cor:evenform}
Let $f(\x)\in \bbH_d[\x]$. 
\begin{enumerate}
	\item[(i)] If $d$ is even, then a PSDC decomposition is
	\begin{equation}\label{eq:evenform}
		f(\x)=\sum_{\balpha\in \calI^{+}}\lambda_{\balpha}\langle \balpha,\x\rangle^d-\sum_{\balpha\in \calI^{-}}(-\lambda_{\balpha})\langle \balpha,\x\rangle^d,
	\end{equation}
	where $\calI^{+}=\{\balpha\in \calI_{n,d}: \lambda_{\balpha}>0\}$ and $\calI^{-}=\{\balpha\in \calI_{n,d}: \lambda_{\balpha}<0\}.$
	\item[(ii)] If $d$ is odd, then a PSDC decomposition is
	\begin{equation}\label{eq:oddform}
		f(\x)=\sum_{\balpha\in \calI^{+}}\lambda_{\balpha}\langle \balpha,\hat{\x}\rangle^{d+1}-\sum_{\balpha\in \calI^{-}}(-\lambda_{\balpha})\langle \balpha,\hat{\x}\rangle^{d+1},
	\end{equation}
	where $\calI^{+}=\{\balpha\in \calI_{n+1,d+1}: \lambda_{\balpha}>0\}, \calI^{-}=\{\balpha\in \calI_{n+1,d+1}: \lambda_{\balpha}<0\}$ and $\hat{\x}=\begin{pmatrix}
\x\\
1
\end{pmatrix}.$
\end{enumerate}
\end{proposition}
\begin{proof}
Case (i) is elucidated by Lemma \ref{lem:psrep}. For the proof of case (ii), we multiply $f(\x)$ with an additional variable $x_h$, yielding an even degree form  $\tilde{f}(\x,x_h)=x_hf(\x)$. Subsequently, case (ii) is deduced by applying case (i) to $\tilde{f}(\x,x_h)$ and setting $x_h=1$.
\qed
\end{proof}
For any polynomial $f(\x)\in\R_d[\x]$, we can express it as:
$$f(\x)=\sum\nolimits_{k=0}^d f_k(\x)$$
where each $f_k(\x)\in\bbH_k[\x]$ for $k=0,1,\ldots,d.$ We can then apply Proposition \ref{cor:evenform} to each form $f_k(\x), k=1,2,\ldots,d$, obtaining a so-called \emph{termwise power-sum-DC} (T-PSDC) decomposition of $f(\x)$. The decomposition procedure is summarized in Algorithm \ref{alg:T-PSDC}. 

 \begin{algorithm}[ht]
\caption{\textbf{T-PSDC Decomposition}}
\label{alg:T-PSDC}
\KwIn{$f_k(\x),k=1,\ldots,d$ (all homogeneous components of the polynomial $f(\x)$);}
\KwOut{DC components $g_k(\x)$ and $h_k(\x), k=1,\ldots,d$;}
\For{$k=1,\ldots,d$}{
\eIf{$k$ is odd}{
$\hat{n}\leftarrow n+1,\hat{d}\leftarrow k+1,\hat{\x}\leftarrow\begin{pmatrix}
\x\\
1
\end{pmatrix};$\\
}{$\hat{n}\leftarrow n,\hat{d}\leftarrow k,\hat{\x}\leftarrow \x;$\\}
Rewrite $f_k(\x)=\sum\nolimits_{\balpha\in\calI_{\hat{n},\hat{d}}}c_{\balpha}\x^{\alpha}$;\\
$\calI \leftarrow \calI_{\hat{n},\hat{d}}$;\\
$\bs{c}\leftarrow$ the coefficients of $f_k(\x)$;\\
Compute $\bs{\lambda}$ by solving linear equations: 
$\whV^{\top}(\hat{n},\hat{d})\bs{\lambda}=\bs{c}$;\\
Set $\calI^{+}=\{\balpha\in \calI: \lambda_{\balpha}>0\},\calI^{-}=\{\balpha\in \calI: \lambda_{\balpha}<0\};$\\
Compute $g_k(\x)=\sum_{\balpha\in \calI^{+}}\lambda_{\balpha}\langle \balpha,\hat{\x}\rangle^{\hat{d}}$ and $h_k(\x)=\sum_{\balpha\in \calI^{-}}-\lambda_{\balpha}\langle \balpha,\hat{\x}\rangle^{\hat{d}};$\\
}
\Return $g_k(\x)$ and $h_k(\x), k=1,\ldots,d$;\\
\end{algorithm}

\paragraph{\textbf{HD-PSDC decomposition}}: 
An $n$-variate $d$-degree form can be dehomogenized to a polynomial in 
$n-1$ variables of degree at most $d$ by assigning a constant value of $1$ to one variable (cf. dehomogenization). Conversely, any polynomial can be transformed into a form by introducing a new variable $x_h$ and adjusting each monomial with appropriate powers of $x_h$, as expressed by $f_H(\x,x_h)=x_h^d\ f(x_1/x_h,\ldots,x_n/x_h)$ (cf. homogenization). Consequently, we propose the \emph{homogenizing-dehomogenizing power-sum-DC} (HD-PSDC) decomposition as encapsulated in Algorithm \ref{alg:HD-PSDC}, founded on Proposition \ref{cor:oddevenhomogenization}.
\begin{proposition}\label{cor:oddevenhomogenization}
Let $f(\x)\in \R_d[\x]$. One has the PSDC decomposition
\begin{equation}\label{eq:oddevenhomogenization}
f(\x)=\sum_{\balpha\in \calI^{+}}\lambda_{\balpha}\langle \balpha,\hat{\x}\rangle^{d_f}-\sum_{\balpha\in \calI^{-}}(-\lambda_{\balpha})\langle \balpha,\hat{\x}\rangle^{d_f},
\end{equation}
where $\calI^{+}=\{\balpha\in \calI_{n+1,d_f}: \lambda_{\balpha}>0\}, \calI^{-}=\{\balpha\in \calI_{n+1,d_f}: \lambda_{\balpha}<0\}$ and $d_f=2\lceil\frac{d}{2}\rceil\footnote{${\lceil \cdot \rceil}$ is the ceiling function of a number.},\hat{\x}=\begin{pmatrix}
\x\\
1
\end{pmatrix}.$
\end{proposition}
 \begin{proof}
Let $f(\x)\in \mathbb{R}_d[\x]$. We will establish the result through two distinct cases:
\begin{enumerate}
    \item[(i)] For an even $d$, we homogenize $f(\x)$ to yield an $(n+1)$-variate polynomial of degree $d$ as $$f_H(\x,x_h)=x_h^d\ f(x_1/x_h,\ldots,x_n/x_h).$$
    Then, we apply Lemma \ref{lem:psrep} for $f_{H}(\x,x_h)$ and assign $x_h=1$, resulting in  the desired expression.
    \item[(ii)] For an odd $d$, we homogenize $f(\mathbf{x})$ to attain an $(n+1)$-variate polynomial of degree $d+1$ as $$f_H(\x,x_h)=x_h^{d+1}\ f(x_1/x_h,\ldots,x_n/x_h).$$
    Again, we apply Lemma \ref{lem:psrep} for $f_{H}(\x,x_h)$ and set $x_h=1$, yielding the desired result.
\end{enumerate}
\qed\end{proof}

\begin{algorithm}[ht]
\caption{HD-PSDC Decomposition}
\label{alg:HD-PSDC}
\KwIn{polynomial $f(\x)$;}
\KwOut{DC components $g(\x)$ and $h(\x)$;}
$d\leftarrow$ the degree of $f(\x)$;\\
\If{$d$ is odd}{
$d\leftarrow d+1$;\\
}
 $n\leftarrow n+1, \hat{\x}\leftarrow \begin{pmatrix}
\x\\
1
\end{pmatrix}$;\\
Rewrite $f(\x)=\sum\nolimits_{\balpha\in\calI_{n,d}}c_{\balpha}\x^{\alpha}$;\\
$\calI\leftarrow{\calI_{n,d}}$;\\
$\bs{c}\leftarrow$ the coefficients of $f(\x)$;\\
 Compute $\bs{\lambda}$ by solving the linear system: 
 $\whV^{\top}(n,d)\bs{\lambda}=\bs{c};$\\
Set $\calI^{+}=\{\balpha\in \calI: \lambda_{\balpha}>0\},\calI^{-}=\{\balpha\in \calI: \lambda_{\balpha}<0\};$\\
Compute $g(\x)=\sum_{\balpha\in \calI^{+}}\lambda_{\balpha}\langle \balpha,\hat{\x}\rangle^{d}$ and $h(\x)=\sum_{\balpha\in \calI^{-}}-\lambda_{\balpha}\langle \balpha,\hat{\x}\rangle^d;$\\
 \Return $g(\x)$ and $h(\x)$.  
\end{algorithm}

\paragraph{Numerical test:} We compare the proposed two PSDC decompositions and the DSOS-DC decomposition in \cite{ahmadi2018dc}. All algorithms are implemented  on MATLAB using the polynomial optimization library SPOT \cite{megretski2013spot} and tested with the solver MOSEK \cite{aps2019mosek} and $\text{MATLAB}$ R2021a on a laptop equipped with Intel Core i5-1035G1 CPU 1.19GHz and 8GB RAM. The numerical results are summarized in Table \ref{tab:twoPSDC}. 

We observed that for odd degree cases $d\in \{3,5\}$, the total running time for T-PSDC and HD-PSDC are almost the same, whereas for even degree cases $d\in\{4,6\}$, HD-PSDC performs about $2$ times faster than T-PSDC. Moreover, T-PSDC and HD-PSDC are at least $3$ times faster than DSOS-DC for all cases. The total running time increases slowly with respect to the increase of $n$, while sharply with respect to the increase of $d$. Hence, for higher degree cases with $d \ge 5$, all three methods suffer from the out-of-memory issue for some large $n$ on MATLAB. Furthermore, the running time of HD-PSDC and T-PSDC mainly depends on the generating time of the power-sum matrix, which is approximately $10\sim 100$ times greater than solving the sparse linear system. 
\begin{sidewaystable}[htbp!]
\centering
\resizebox{\linewidth}{!}{
    \begin{tabular}{l|c|ll|ll|ll|ll|ll|ll|ll}
    \toprule
    \multirow{2}*{Alg.} &\multirow{2}*{\diagbox {d}{n}}&\multicolumn{2}{c|}{6}&\multicolumn{2}{c|}{15}&\multicolumn{2}{c|}{24}&\multicolumn{2}{c|}{33}&\multicolumn{2}{c|}{42}&\multicolumn{2}{c|}{51}&\multicolumn{2}{c}{60}\\
    \cline{3-16}&&T$_{gen}$&T$_{sol}$&T$_{gen}$&T$_{sol}$&T$_{gen}$&T$_{sol}$&T$_{gen}$&T$_{sol}$&T$_{gen}$&T$_{sol}$&T$_{gen}$&T$_{sol}$&T$_{gen}$&T$_{sol}$\\
    \midrule
    \multirow{4}*{T-PSDC}&3&0.01&0.00&0.14&0.01&1.69&0.07&13.72&0.25&74.90&0.77&318.35&1.74&1121.43&3.29\\
                        &4&0.01&0.00&0.22&0.02&2.94&0.13&24.46&0.46&135.68&1.53&577.79&3.36&2030.30&6.70\\
                        &5&0.04&0.02&12.35&1.96&940.43&35.01&--&--&--&--&--&--&--&--\\
                        &6&0.05&0.02&19.43&3.53&1547.60&59.82&--&--&--&--&--&--&--&--\\
    \hline
     \multirow{4}*{HD-PSDC}&3&0.00&0.01&0.13&0.01&1.65&0.07&13.60&0.24&74.55&0.77&319.52&1.73&1120.31&3.24\\
                          &4&0.00&0.00&0.12&0.01&1.65&0.08&13.43&0.24&71.85&0.80&300.62&1.85&1032.44&3.59\\
                          &5&0.02&0.01&12.31&2.04&926.45&35.27&--&--&--&--&--&--&--&--\\
                          &6&0.03&0.01&12.21&1.94&945.31&34.73&--&--&--&--&--&--&--&--\\
     \hline
     \multirow{4}*{DSOS-DC}&3&\multicolumn{2}{c|}{0.74}&\multicolumn{2}{c|}{5.32}&\multicolumn{2}{c|}{36.75}&\multicolumn{2}{c|}{156.66}&\multicolumn{2}{c|}{508.64}&\multicolumn{2}{c|}{1062.69}&\multicolumn{2}{c}{--}\\
                            &4&\multicolumn{2}{c|}{0.28}&\multicolumn{2}{c|}{5.71}&\multicolumn{2}{c|}{39.82}&\multicolumn{2}{c|}{198.58}&\multicolumn{2}{c|}{567.10}&\multicolumn{2}{c|}{1218.55}&\multicolumn{2}{c}{--}\\
                            &5&\multicolumn{2}{c|}{3.70}&\multicolumn{2}{c|}{2835.44}&\multicolumn{2}{c|}{--}&\multicolumn{2}{c|}{--}&\multicolumn{2}{c|}{--}&\multicolumn{2}{c|}{--}&\multicolumn{2}{c}{--}\\
                             &6&\multicolumn{2}{c|}{3.75}&\multicolumn{2}{c|}{2881.09}&\multicolumn{2}{c|}{--}&\multicolumn{2}{c|}{--}&\multicolumn{2}{c|}{--}&\multicolumn{2}{c|}{--}&\multicolumn{2}{c}{--}\\
     \bottomrule
    \end{tabular}}
    \caption{Comparison of three DC decomposition algorithms (T-PSDC, HD-PSDC and DSOS-DC) for polynomials with $d\in\{3,4,5,6\}$, $n\in \{6,15,\ldots,60\}$, and uniformly distributed random integer coefficients sampled in $[-10,10]$. For each pair $(n,d)$, we perform $10$ runs of each algorithm and compare the average CPU time (in seconds). For T-PSDC and HD-PSDC, T$_{gen}$ and T$_{sol}$ denote the average running time for generating the matrix $\whV$ and for solving the linear system \eqref{eq:linear-system} respectively. For DSOS-DC, only the average total running time is reported.}
    \label{tab:twoPSDC}
\end{sidewaystable}

\subsection{DC reformulations of \eqref{eq:POPL} leveraging PSDC decompositions}
Using the PSDC decompositions T-PSDC and HD-PSDC, as described in Propositions \ref{cor:evenform} and \ref{cor:oddevenhomogenization}, we transform \eqref{eq:POPL} into two DC programs as follows:\\
Let $|\cdot|_p$ represent the $\ell_p$ norm. Then
\begin{itemize}
	\item [(i)] Using T-PSDC decomposition, problem \eqref{eq:POPL} can be reformulated as 
	\begin{equation}\label{eq:T-PSDC}
		\min_{\x\in\calP} f(\x)=\underbrace{\sum_{p=1}^{\lceil \frac{d}{2} \rceil}\|\bs{A}_p^{+}\x+\bs{b}_p^{+}\|_{2p}^{2p}}_{g_1(\x)}-\underbrace{\sum_{p=1}^{\lceil \frac{d}{2} \rceil}\|\bs{A}_p^{-}\x+\bs{b}_p^{-}\|_{2p}^{2p}}_{h_1(\x)},\tag{T--P$_\text{dc}$}
	\end{equation}
	where $\bs{A}_p^+\in\R^{m_1\times n},\bs{A}_p^-\in\R^{m_2\times n},\bs{b}_p^{+}\in\R^{m_1},\bs{b}_p^{-}\in\R^{m_2}$ with $$m_1+m_2\le s_{n,2p}+s_{n+1,2p}$$ and $p=1,\ldots,\lceil \frac{d}{2} \rceil.$
	\item [(ii)] Using HD-PSDC decomposition, problem \eqref{eq:POPL} can be reformulated as 
	\begin{equation}\label{eq:HD-PSDC}
		\min_{\x\in\calP} f(\x)=\underbrace{\|\bs{A}^{+}\x+\bs{b}^{+}\|_{d_f}^{d_f}}_{g_2(\x)}-\underbrace{\|\bs{A}^{-}\x+\bs{b}^{-}\|_{d_f}^{d_f}}_{h_2(\x)},\tag{HD--P$_\text{dc}$}
	\end{equation}
	where $d_f=2\lceil \frac{d}{2} \rceil$ and $\bs{A}^+\in\R^{m_1\times n},\bs{A}^-\in\R^{m_2\times n},\bs{b}^{+}\in\R^{m_1},\bs{b}^{-}\in\R^{m_2}$ with $$m_1+m_2\le s_{n+1,d_f}.$$
\end{itemize}		

\section{BDCA$_\text{e}$ for solving \eqref{eq:T-PSDC} and \eqref{eq:HD-PSDC} }\label{sec:DCAe}
Problem \eqref{eq:T-PSDC} and \eqref{eq:HD-PSDC} are a subclass of DC programming. The classical DCA \cite{1997Convex} is viable and efficient for solving these problems. Meanwhile, Niu et al. proved in \cite[Proposition 1]{niu2019higher} that the direction $\bs{d}^k=\y^k-\x^k$ in DCA is a descent direction (namely, DC descent direction) of $f$ at $\bs{y}^k$ if (1) $\bs{d}^k\neq 0$, (2) $\bs{d}^k$ is a feasible direction of $\calP$ at $\y^k$, and (3) $h$ is strongly convex. In this scenario, executing a line search (either exact or inexact) along $\bs{d}^k$ will yield an additional reduction in the value of the objective function, thereby ensuring an acceleration of DCA. Subsequently, we introduce a Boosted DCA with Exact Line Search (BDCA$_\text{e}$) for problems \eqref{eq:T-PSDC} and \eqref{eq:HD-PSDC}, as delineated in Algorithm \ref{alg:BDCAe}.
\begin{algorithm}[ht!]
	\caption{BDCA$_\text{e}$ for solving problem \eqref{eq:T-PSDC} (or \eqref{eq:HD-PSDC})}
	\label{alg:BDCAe}
	\KwIn{$\x^0\in\calP$, $\varepsilon>0$;}
	\For{$k=0,1,\ldots$}{
	\textbf{Step 1:} Solving the subproblem 
		\begin{equation}\label{eq:subproblem}
		\y^k\in\argmin_{\x\in\calP} ~g_i(\x)-\langle\nabla h_i(\x^k),\x\rangle, \tag{P$_k$}
		\end{equation}
		where $i=1$ for \eqref{eq:T-PSDC} and $i=2$ for \eqref{eq:HD-PSDC}.\\
		\textbf{Step 2:}
		Set $\bs{d}^k=\y^k-\x^k$. If $\|\bs{d}^k\|/(1+\|\x^k\|)<\varepsilon,$ \Return $\x^k$;\\
		\textbf{Step 3:} Set  $\mathcal{I}(\bs{d}^k)=\{i\in\{1,\ldots,m\}:\langle\bs{a}_i,\bs{d}^k\rangle>0\}$ and compute $$\bar{t}=
			\min_{i\in\mathcal{I}(\bs{d}^k)}\frac{b_i-\langle\bs{a}_i,\y^k\rangle}{\langle\bs{a}_i,\bs{d}^k\rangle}.$$\\
		\textbf{Step 4:}
		If $\bar{t}>0$, then $\bs{d}^k$ is a feasible direction, we compute $$t_k=\text{ExactLineSearch}(\bs{d}^k,\y^k,\bar{t}) \text{ by Algorithm \ref{alg:E-line-search}},$$ Otherwise, set 
		$t_k=0$;\\
		\textbf{Step 5:}
		Set $\x^{k+1}=\y^k+t_k\bs{d}^k$;
		}
\end{algorithm}

Here are some comments on Algorithm \ref{alg:BDCAe}:
\begin{itemize}
    \item [(i)] BDCA$_\text{e}$ reduces to the classical DCA when excluding \textbf{Step 3--4} and setting $t_k=0$ for all $k$ in \textbf{Step 5}. Additionally, if the exact line search in \textbf{Step 4} is replaced by the Armijo line search, then BDCA$_\text{e}$ reduces to BDCA in \cite{niu2019higher}.
    \item[(ii)] In \textbf{Step 1}, we leverage the fast dual-based proximal gradient (FDGP) method \cite{beck2014fast} to solve \eqref{eq:subproblem} by taking advantage of the power-sum structure, which will be discussed in Section \ref{sec:FDPG}.
    \item[(iii)] In \textbf{Step 3--4}, we introduce the exact line search to accelerate DCA, which will be discussed in the next two paragraphs. In particular, our exact line search amounts to computing the roots of the derivative of a univariate polynomial, whose coefficients can be computed explicitly by Equation \eqref{eq:c-j} based on the power-sum formulation. Note that the assumption that $f$ is lower bounded over $\calP$ guarantees the finiteness of $t_k$. Otherwise, when $\calI(\bs{d}^k)=\emptyset$, then $\bar{t}=\min \emptyset = \infty$ by convention and it is possible to find $t_k=\infty.$ In this case,  \eqref{eq:P-dc} has no optimal solution.
\end{itemize}
\begin{algorithm}[ht!]
	\caption{Exact line search}
	\label{alg:E-line-search}
	\KwIn{A descent direction $\bs{d}^k$, a point $\y^k$ and a bound $\bar{t}$;}
	\KwOut{Optimal step size $t_k$;}
	Compute $c_j (j=0,\ldots,d_f)$ by Equation \eqref{eq:c-j};\\
	Set $\mathcal{R}\leftarrow\{t:\hat{f}^{\prime}(t)=0\}\cup\{0,\bar{t}\}$;\\
	$t_k\in\argmin_{t\in\mathcal{R}}\hat{f}(t);$\\
\end{algorithm}

In the following, we will explore various facets of BDCA$_\text{e}$, encompassing: the verification of the feasibility of direction $\bs{d}^k$ at $\y^k$ over $\calP$, the estimation of an upper bound for the line search step size, and the elucidation of the exact line search procedure.
\paragraph{\textbf{Upper bound estimation for the line search step size}}:
The next proposition shows that
	\begin{equation}
		\label{eq:tbar}
		\bar{t}=\min_{i\in\mathcal{I}(\bs{d}^k)}\left\{\frac{b_i-\langle\bs{a}_i,\y^k\rangle}{\langle\bs{a}_i,\bs{d}^k\rangle}\right\}
	\end{equation} 
is an upper bound for the line search step size $t_k$ to ensure that $\y^k+t_k\bs{d}^k\in \calP$, where $\mathcal{I}(\bs{d}^k)=\{i\in\{1,\ldots,m\}:\langle\bs{a}_i,\bs{d}^k\rangle>0\}.$
\begin{lemma}\label{lem:delta}
Let $\y^k$ and $\bs{d}^k$ be generated in Algorithm \ref{alg:BDCAe}. Then $\bs{d}^k\in \mathcal{D}(\y^k)$ is equivalent to $\bar{t}>0$. Moreover, if $t \in [0,\bar{t}]$, then $\y^k+t\bs{d}^k\in \calP$, and if $t >\bar{t}$, then $\y^k+t\bs{d}^k\notin \calP$.
\end{lemma}
\begin{proof}
We first prove the following result:
$$\bs{d}^k\in\mathcal{D}(\y^k)~~\Leftrightarrow~~ \min_{i\in\mathcal{I}(\bs{d}^k)}\{b_i-\langle\bs{a}_i,\y^k\rangle\}>0.$$
If $\mathcal{I}(\bs{d}^k)=\emptyset$, the conclusion is obvious by the definition of the cone of feasible directions $\mathcal{D}(\y^k)$. \\
If $\mathcal{I}(\bs{d}^k)\ne\emptyset$, then for $i\in \{1,\ldots,m\}$ and $j\in\{1,\ldots,r\}$, we have  
	\begin{equation*}
		\bs{d}^k\in\mathcal{D}(\y^k)~\Leftrightarrow~ 
		\begin{cases}
			\langle\bs{a}_i,\y^k\rangle=b_i\Rightarrow \langle\bs{a}_i,\bs{d}^k\rangle\le0,\\
			\langle\bs{p}_j,\bs{d}^k\rangle=0.
		\end{cases}
	\end{equation*}
	Invoking the fact that $\langle\bs{p}_j,\bs{d}^k\rangle=\langle\bs{p}_j,\y^k\rangle-\langle\bs{p}_j,\x^k\rangle=q_j-q_j=0$, we deduce 
	$$\bs{d}^k\in\mathcal{D}(\y^k)~\Leftrightarrow~ \langle\bs{a}_i,\y^k\rangle=b_i\Rightarrow \langle\bs{a}_i,\bs{d}^k\rangle\le0~\Leftrightarrow~\langle\bs{a}_i,\bs{d}^k\rangle>0\Rightarrow\langle\bs{a}_i,\y^k\rangle<b_i.$$
The last equivalence condition implies that the conclusion holds. \\
We now proceed to prove the main result as the following:\\
	(i) If $\mathcal{I}(\bs{d}^k)=\emptyset$, we have $\langle\bs{a}_i,\bs{d}^k\rangle\le0$ for $i=1,2,\ldots,m$. Thus, $\forall t\in[0,\infty)$, we observe that the fact
	\begin{equation*}
		\begin{cases}
			\langle\bs{a}_i,\y^k+t\bs{d}^k\rangle=\langle\bs{a}_i,\y^k\rangle+t\langle\bs{a}_i,\bs{d}^k\rangle\le b_i,\\
			\langle\bs{p}_j,\y^k+t\bs{d}^k\rangle=\langle\bs{p}_j,\bs{y}^k\rangle+t\langle\bs{p}_j,\bs{d}^k\rangle=q_j.
		\end{cases}
	\end{equation*}
	holds for $i=1,2,\ldots,m$ and $j=1,2,\ldots,r$. Therefore, we have $\y^k+t\bs{d}^k\in\calP$.\\
	(ii) If $\mathcal{I}(\bs{d}^k)\ne\emptyset$, invoking the fact that $\bar{t}=\min_{i\in\mathcal{I}(\bs{d}^k)}\{\frac{b_i-\langle\bs{a}_i,\y^k\rangle}{\langle\bs{a}_i,\bs{d}^k\rangle}\}>0$, hence, $\forall t\in[0,\bar{t}]$, we have 
	\begin{equation*}
		\langle\bs{p}_j,\y^k+t\bs{d}^k\rangle=\langle\bs{p}_j,\bs{y}^k\rangle+t\langle\bs{p}_j,\bs{d}^k\rangle=q_j.
	\end{equation*} for $j=1,\ldots,r$.\\
	Next, we verify the result $\langle\bs{a}_i,\y^k+t\bs{d}^k\rangle\le b_i$ for $i=1,\ldots,m$ in two cases:\\
	\textbf{Case 1:} If $i\in\mathcal{I}(\bs{d}^k)$, it implies $\langle\bs{a}_i,\bs{d}^k\rangle>0$, we get that  
	$$\langle\bs{a}_i,\y^k+t\bs{d}^k\rangle=\langle\bs{a}_i,\y^k\rangle +t\langle\bs{a}_i,\bs{d}^k\rangle\le\langle\bs{a}_i,\y^k\rangle +\min_{i\in\mathcal{I}(\bs{d}^k)}\{\frac{b_i-\langle\bs{a}_i,\y^k\rangle}{\langle\bs{a}_i,\bs{d}^k\rangle}\}\langle\bs{a}_i,\bs{d}^k\rangle\le b_i.$$
	\textbf{Case 2:} If $i\notin\mathcal{I}(\bs{d}^k)$, it implies $\langle\bs{a}_i,\bs{d}^k\rangle\le0$. Thus, for all $i\notin\mathcal{I}(\bs{d}^k)$, we have $$\langle\bs{a}_i,\y^k+t\bs{d}^k\rangle=\langle\bs{a}_i,\y^k\rangle+t\langle\bs{a}_i,\bs{d}^k\rangle\le b_i.$$
	Moreover, $\forall t_0>\bar{t}$ (Here, we always assume that $\bar{t}<\infty$), we have some $i_0\in\mathcal{I}(\bs{d}^k)$ such that  $$t_0>\frac{b_{i_0}-\langle\bs{a}_{i_0},\y^k\rangle}{\langle\bs{a}_{i_0},\bs{d}^k\rangle}.$$
	Therefore, we can deduce that $$\langle\bs{a}_{i_0},\y^k+t_0\bs{d}^k\rangle=\langle\bs{a}_{i_0},\y^k+\delta_0\bs{d}^k\rangle=\langle\bs{a}_{i_0},\y^k\rangle+\delta_0\langle\bs{a}_{i_0},\bs{d}^k\rangle>b_{i_0}.$$
	This implies $\y^k+t_0\bs{d}^k\notin\calP.$  
\qed
\end{proof}

\paragraph{\textbf{Exact line search}}:
The exact line search procedure in \textbf{Step 4} is detailed in Algorithm \ref{alg:E-line-search}. This algorithm is founded on the minimization problem \eqref{eq:els}:
\begin{equation}\label{eq:els}
	\min_{t\ge0} \{f(\y^k+t\bs{d}^k):\y^k+t\bs{d}^k\in\calP\}. \tag{ELS}
\end{equation}
Without loss of generality, we assume that $f(\x)$ can be represented as Proposition \ref{cor:oddevenhomogenization}, given by the equation:
\begin{equation}\label{eq:els-rp}
	f(\x)=\sum_{\balpha\in\calI_{n+1,d_{f}}}\lambda_{\balpha}
	\langle\balpha,\hat{\x}\rangle^{d_f}.
\end{equation}
Subsequently, substituting $\hat{\x}=\hat{\y}^k+t\hat{\bs{d}}^k$ into \eqref{eq:els-rp} with $\hat{\y}^k
 =\begin{pmatrix}
 \y^k\\
 1
\end{pmatrix}$ and $\hat{\bs{d}}^k=\begin{pmatrix}
 \bs{d}^k\\
 0
\end{pmatrix}$, we obtain 
\begin{equation*}
	\hat{f}(t)=f(\y^k+t\bs{d}^k)=\sum_{j=0}^{d_f}c_jt^j,
\end{equation*}
where the coefficients $c_j$ are determined by
\begin{equation}\label{eq:c-j}
	c_j=\sum_{i=1}^{s_{n+1,d_f}}\lambda_i\tbinom{d_f}{j}\langle\balpha^i, \hat{\y}^k\rangle^{d_f-j}\langle\balpha^i, \hat{\bs{d}}^k\rangle^j.
\end{equation}
Combined with Lemma \ref{lem:delta}, problem \eqref{eq:els} can be represented as the following univariate polynomial minimization problem:
\begin{equation}\label{eq:uni-op}
	\min_t\{\hat{f}(t):0\le t\le\bar{t}\}.
\end{equation}
The minimization problem \eqref{eq:uni-op} is well-defined. Specifically, if the set $\mathcal{I}(\bs{d}^k)$ is non-empty, it implies that $\bar{t}$ is finite, thus confirming the validity of the assertion. Conversely, if $\mathcal{I}(\bs{d}^k)$ is empty, then $\bar{t}$ assumes an infinite value, and the assertion remains valid due to the property that $|\hat{f}(t)|\to \infty$ as $t\to \infty$ and the assumption that $f$ is lower bounded over $\calP$.
\begin{remark} 
Contrary to the Armijo line search method elaborated upon in \cite{artacho2018accelerating,niu2019higher}, our suggested technique necessitates the computation of the parameter $c_j$ as illustrated in Equation \eqref{eq:c-j}, and the resolution of Problem \eqref{eq:uni-op}, diverging from the traditional computation of objective function values. The minimization of Problem \eqref{eq:uni-op} can be effectively executed by identifying the roots of the derivative $\hat{f}^{\prime}(t)$ of $\hat{f}(t)$. Particularly, when $d_f\le 5$, the roots of $\hat{f}^{\prime}(t)$ can be efficiently computed using a root-finding formula, whereas for instances where $d_f>5$, a numerical method is a viable alternative. This methodology allows us to derive an optimal step size for each iteration of the line search procedure.
\end{remark}

\section{Convergence analysis for BDCA$_\text{e}$}\label{sec:con-analysis}
\begin{assume}\label{assumption:h}
	The component $h$ is $\rho$-strongly convex with $\rho>0$.
\end{assume}
Note that this assumption is easy to satisfy for problems \eqref{eq:T-PSDC} and $\eqref{eq:HD-PSDC}$. Indeed, for a DC function $f=\tilde{g}-\tilde{h}$, we can take $g:=\tilde{g}+\frac{\rho}{2}\|\x\|^2$ and $h:=\tilde{h}+\frac{\rho}{2}\|\x\|^2$ with a small $\rho>0$ to ensure the strong convexity of $h$.  
\begin{lemma}\cite{artacho2018accelerating,niu2019higher}\label{lem:f-descent} Under Assumption \ref{assumption:h}, let $\{\y^k\}$, $\{\x^k\}$ and $\{\bs{d}^k\}$ be the sequences generated by Algorithm \ref{alg:BDCAe} when applied to problem \eqref{eq:P-dc}, it holds that
	\begin{equation}\label{eq:f-descent}
		f(\y^k)\le f(\x^k)-\rho\|\bs{d}^k\|^2.    
	\end{equation}
\end{lemma}
\begin{theorem}\label{thm:convergence}
	Under Assumption \ref{assumption:h}, if the sequence $\{\x^k\}$ generated by BDCA$_\text{e}$ in Algorithm \ref{alg:BDCAe} when applied to problem \eqref{eq:P-dc} is bounded, then
	\begin{itemize}
		\item[(i)] The sequence $\{f(\x^k)\}$ is decreasing and converging to some finite $f_{opt}$.
		\item [(ii)] Every limit point of the sequence $\{\x^k\}$ is a critical point of problem \eqref{eq:P-dc}.
		\item[(iii)]$\sum_{k=0}^{\infty}\|\bs{d}^k\|^2<\infty$.
	\end{itemize}
\end{theorem}
\begin{proof}(i)To simplify the notations, we use $g$ (resp. $h$) instead of $g_i$ (resp. $h_i$) for $i\in\{1,2\}$ in Algorithm \ref{alg:BDCAe}. Invoking \textbf{Step 4} and \textbf{Step 5} of Algorithm \ref{alg:BDCAe}, we deduce that
	$$f(\x^{k+1})\le f(\y^k).$$
	Combining the previous inequality and \eqref{eq:f-descent}, we derive that
	\begin{equation}\label{eq:d-fxk}
		f(\x^{k+1})\le f(\x^k)-\rho\|\bs{d}^k\|^2,
	\end{equation}
as $\{f(\x^k)\}$ is decreasing and $f$ is lower bounded over $\calP$, 
we conclude (i).\\
(ii) By the boundedness of $\{\x^k\}$, we know that convergent subsequences exist. Let $\x^*$ be a limit point of $\{\x^k\}$ and $\{\x^{k_j}\}$ be a subsequence of $\{\x^k\}$ converging to $\x^*$. As $k_j\rightarrow \infty$, we infer from  \eqref{eq:d-fxk} that 
$$\|\y^{k_j}-\x^{k_j}\|\to 0.$$
this, combined with the fact that $\|\y^{k_j}-\x^*\|\le\|\y^{k_j}-\x^{k_j}\|+\|\x^{k_j}-\x^*\|$, yields $\y^{k_j}\to \x^*$.
By invoking the fact that $$\y^{k_j}\in\argmin_{\x}\{g(\x)-\langle\nabla h(\x^{k_j}),\x\rangle,\x\in\calP\},$$ 
we can conclude, for all $\x\in\calP$, that:
$$\langle\nabla g(\y^{k_j})-\nabla h(\x^{k_j}),\x-\y^{k_j}\rangle\ge 0.$$
Taking $k_j\to \infty$, and noting that $g$ and $h$ are continuously differentiable, we obtain
	$$\langle\nabla f(\x^{*}),\x-\x^{*}\rangle\ge 0.$$
This implies that $\x^*$ is a critical point for problem \eqref{eq:P-dc}.\\
(iii)Summing inequality \eqref{eq:d-fxk} from $k=0$ to $\infty$, we obtain
	$$\sum_{k=0}^{\infty}\rho\|\bs{d}^k\|^2\le\sum_{k=0}^{\infty} ( f(\x^k)-f(\x^{k+1}))=f(\x^0)-\lim_kf(\x^{k+1})<\infty,$$
	therefore, we can conclue that  $\sum_{k=0}^{\infty}\|\bs{d}^k\|^2<\infty.$ 
	\qed
\end{proof}

\begin{remark}
Note that if we assume that the line search step size $t_k$ is upper bounded for all $k\ge 0$, then we will have a stronger convergence result based on \cite[Theorem 11]{niu2019higher} that the sequence $\{\x^k\}$ in BDCA$_\text{e}$ is also convergent.
\end{remark}

By employing the following useful lemma, we can obtain the rate of convergence of the sequence $\{f(\x^k)\}$ generated by Algorithm \ref{alg:BDCAe}. 
\begin{lemma}\label{lem:sk}
	Let $\{s_k\}$ be a nonincreasing and nonnegative real sequence converging to $0$. Suppose that there exist $\alpha\geq 0$ and $\beta>0$ such that for all large enough $k$,
	\begin{equation}
		\label{eq:seqrel-bis}
		s_{k+1}^{\alpha} \leq \beta (s_k-s_{k+1}).
	\end{equation}
	Then
	\begin{enumerate}
		\item[(i)] if $\alpha=0$, then the sequence $\{s_k\}$ converges to $0$ in a finite number of steps;
		\item[(ii)] if $\alpha\in (0,1]$, then the sequence $\{s_k\}$ converges linearly to $0$ with rate $\frac{\beta}{1+\beta}$.
		\item[(iii)] if $\alpha>1$, then the sequence $\{s_k\}$ converges sublinearly to $0$, i.e., there exists $\eta>0$ such that 
		\begin{equation}
			\label{eq:cvord-bis}
			s_k\leq \eta k^{\frac{1}{1-\alpha}}
		\end{equation}
		for large enough $k$.
	\end{enumerate}
\end{lemma}
\begin{proof}
	(i) If $\alpha=0$, then \eqref{eq:seqrel-bis} implies that
	$$0\leq s_{k+1}\leq s_k - \frac{1}{\beta}.$$
	It follows by $s_k\to 0$ and $\frac{1}{\beta}>0$ that $\{s_k\}$ converges to $0$ in a finite number of steps, and we can estimate the number of steps as:
	$$0\leq s_{k+1} \leq s_{k}-\frac{1}{\beta} \leq s_{k-1}-\frac{2}{\beta}\leq \cdots \leq s_{N}-\frac{k-N+1}{\beta}.$$
	Hence $$k\leq \beta s_N+N-1.$$
	(ii) If $\alpha \in (0,1]$. Since $s_k\to 0$, we have that $s_k<1$ for large enough $k$. Thus, $s_{k+1}\leq s_k<1$, and it follows by \eqref{eq:seqrel-bis} that 
	$$s_{k+1} \leq s_{k+1}^{\alpha} \leq \beta (s_k-s_{k+1})$$
	for large enough $k$. Hence
	$$s_{k+1}\leq \left(\frac{\beta}{1+\beta}\right) s_k,$$
	i.e., $\{s_k\}$ converges linearly to $0$ with rate $\frac{\beta}{1+\beta}$ for large enough $k$.\\
	(iii) If $\alpha>1$, then we have two cases:\\
	\textbf{Case 1:} if $\{x^k\}$ converges in a finite number of steps, then the inequality \eqref{eq:cvord-bis} trivially holds. \\
	\textbf{Case 2:} Otherwise, $s_k>0$ for all $k\in \N$. Let $\phi(t)=t^{-\alpha}$ and $\tau > 1$.\\
	(a) Suppose that $\phi(s_{k+1})\leq \tau \phi(s_k)$. By the decreasing of $\phi(t)$ and $s_k\geq s_{k+1}$, we have 
	$$\phi(s_k)(s_k-s_{k+1}) \leq \int_{s_{k+1}}^{s_k} \phi(t)~\mathrm{d} t = \frac{1}{1-\alpha}(s_{k}^{1-\alpha}-s_{k+1}^{1-\alpha}).$$
	It follows from \eqref{eq:seqrel-bis} that 
	$$\frac{1}{\beta} \leq \phi(s_{k+1})(s_k-s_{k+1}) \leq \tau\phi(s_{k})(s_k-s_{k+1})  \leq \frac{\tau}{\alpha-1}(s_{k+1}^{1-\alpha} - s_{k}^{1-\alpha}).$$
	Hence
	\begin{equation}
		\label{eq:recineq01}
		s_{k+1}^{1-\alpha} - s_{k}^{1-\alpha} \geq \frac{\alpha-1}{\beta\tau}
	\end{equation}
	for large enough $k$. \\
	(b) Suppose that $\phi(s_{k+1})\geq \tau \phi(s_k)$. Taking $q:=\tau^{-\alpha^{-1}}\in (0,1)$, then
	$$s_{k+1} \leq q s_k.$$
	Hence $$s_{k+1}^{1-\alpha} \geq q^{1-\alpha} s_k^{1-\alpha}.$$
	It follows that $\exists N>0, \forall k\geq N$:
	$$s_{k+1}^{1-\alpha} - s_{k}^{1-\alpha} \geq (q^{1-\alpha}-1) s_k^{1-\alpha} \geq (q^{1-\alpha}-1) s_N^{1-\alpha}.$$
	In both cases, there exist a constant $\zeta := \min \{(q^{1-\alpha}-1) s_N^{1-\alpha}, \frac{\alpha-1}{\beta\tau}\}$ such that for large enough $k$, we have 
	$$s_{k+1}^{1-\alpha} - s_{k}^{1-\alpha} \geq \zeta.$$
	Summing for $k$ from $N$ to $M-1>N$, we have 
	$$s_{M}^{1-\alpha} - s_{N}^{1-\alpha} \geq \zeta (M-N).$$
	Then, by the decreasing of $t^{\frac{1}{1-\alpha}}$, we get 
	$$s_{M} \leq \left( s_{N}^{1-\alpha} +\zeta (M-N) \right)^{\frac{1}{1-\alpha}}.$$
	Hence, there exist some $\eta>0$ such that 
	$$s_{M}\leq \eta M^{\frac{1}{1-\alpha}}$$
	for large enough $M$, which completes the proof.
\qed\end{proof}

We prove the convergence rate of the sequence $\{f(\x^k)\}$ generated by BDCA$_\text{e}$ as the sequence $\{\x^k\}$ has a limit point $\x^*$ at which $\tilde{f}=f+\chi_{\calP}$ satisfies the Kurdyka--\L ojasiewicz (KL) property \cite{frankel2015splitting} and $\nabla h$ is locally Lipschitz continuous. 
\begin{theorem}\label{thm:f-dk}
Under the assumptions of Theorem \ref{thm:convergence}, let  $\{\x^k\},\{\y^k\}$ be sequences generated by BDCA$_\text{e}$ from a starting point $\x^0\in\calP$, and $\mathcal{V}$ be the set of limit points of $\{\x^k\}$. Suppose that $\tilde{f}:=f+\chi_{\calP}$ satisfies the KL property at $\x^*\in\mathcal{V}$ with $\varphi(s)=Ms^{1-\theta}$ for some $M>0$ and $0\le \theta<1$. Then
the following statements hold:
\begin{enumerate}
    \item [(i)] If $\theta=0$, then the sequence $\{f(\x^k)\}$ converges to $f_{opt}$ in a finite number of steps.
    \item [(ii)] If $\theta \in (0, 1/2]$, then the sequence $\{f(\x^k)\}$ converges linearly to $f_{opt}$, that is, there exist positive constants $N_0, \eta_0$ and $0<q<1$ such that $f(\x^k)-f_{opt}\le \eta_0q^k$ for all $k \ge N_0$.
    \item [(iii)] If $\theta\in (1/2, 1)$, then the sequence $\{f(\x^k)\}$ converges sublinearly to $f_{opt}$, that is, there exist positive constants $\eta_0$ and $N_0$ such that $f(\x^k)-f_{opt}\le \eta_0 k^{-\frac{1}{2\theta-1}} $ for all $k \ge N_0$.
\end{enumerate}
\end{theorem}
\begin{proof} We conclude that $f$ shares the same value over $\mathcal{V}$. As $f$ is continuously differentiable, this follows directly from 
$$f_{opt}=\lim_{k\rightarrow \infty}f(\x^k)=\lim_{j\rightarrow \infty}f(\x^{k_j})=f(\x^*),$$
where the subsequence $\{\x^{k_j}\}\rightarrow \x^*\in\mathcal{V}.$

Given that $\nabla h$ is locally Lipschitz continuous, for each $\x^*\in\mathcal{V}$, there exist some constants $L_{\x^*}\ge0$ and $\delta^{\prime}_{\x^*}>0$ satisfying
\begin{equation}\label{eq:local-lipz}
  \|\nabla h(\x)-\nabla h(\y)\|\le L_{\x^*}\|\x-\y\|,~~\forall \x,\y\in\mathbb{B}(\x^*,\delta^{\prime}_{\x^*}).  
\end{equation}
Furthermore, as $\tilde{f}$ satisfies the Kurdyka--\L ojasiewicz property at $\x^*\in\mathcal{V}$, there exists some constant $\delta^{\prime\prime}_{\x^*}>0$, $\forall \x\in\mathbb{B}(\x^*,\delta^{\prime\prime}_{\x^*})$ with $f(\x^*)<f(\x)<f(\x^*)+\eta$ such that the Kurdyka--\L ojasiewicz inequality holds. Let $\delta_{\x^*}=\min\{\delta^{\prime}_{\x^*},\delta^{\prime\prime}_{\x^*}\},$ we can construct an open cover of the set $\mathcal{V}$: 
$$\mathcal{V}\subset\bigcup_{\x^*\in\mathcal{V}}\mathbb{B}(\x^*,\delta_{\x^*}/4).$$
Since the sequence $\{\x^k\}$ is bounded, then $\mathcal{V}$ is a compact set. Thus, a finite number of points $\bs{v}_1,\ldots,\bs{v}_l\in\mathcal{V}$ exist such that $$\mathcal{V}\subset \bigcup_{i=1}^{l}\mathbb{B}(\bs{v}_i,\delta_{\bs{v}_i}/4).$$ 
Let $L=\max_{1\le i\le l}L_{\bs{v}_i}$ and $\delta=\min_{1\le i\le l}\delta_{\bs{v}_i}/2.$ Invoking the fact from Theorem \ref{thm:convergence} that $\{\x^k\}$ and $\{\y^k\}$ share the same set of limit points $\mathcal{V}$, we obtain $N_1>0$ such that $\{\x^k\}\in\bigcup_{i=1}^{l}\mathbb{B}(\bs{v}_i,\delta_{\bs{v}_i}/2)$ and $\|\x^k-\y^k\|\le\delta$ for all $k\ge N_1$. Thus, for all $k\ge N_1$, there exists some $\bs{v}_i$ such that $\x^k, \y^k\in\mathbb{B}(\bs{v}_i,\delta_{\bs{v}_i}).$ Combined with \eqref{eq:local-lipz}, this implies that for all $k>N_1$
\begin{equation}\label{eq:xk-yk-lipschitz}
    \|\nabla h(\x^k)-\nabla h(\y^k)\|\le L_{\bs{v}_i}\|\x^k-\y^k\|\le L\|\x^k-\y^k\|.
\end{equation}
On the other hand, since $\lim_{k\rightarrow \infty}f(\x^k)=f_{opt},f(\x^k)\ge f_{opt}$, there exist $N_2>0$ and $\eta>0$ such that $f_{opt}<f(\x^k)<f_{opt}+\eta$ for all $k\ge N_2$. Setting $N=\max\{N_1,N_2\}$ and for each $k\ge N$, we conclude that there exists some $\bs{v}_i\in\mathcal{V}$ such that $\x^k, \y^k\in\mathbb{B}(\bs{v}_i,\delta_{\bs{v}_i})$. 
Hence, for each $\y^k\in\mathbb{B}(\bs{v}_i,\delta_{\bs{v}_i})$ and $k\ge N$, we have $$f_{opt}<f(\y^k)\le f(\x^k)<f_{opt}+\eta.$$
Since $\tilde{f}$ satisfies the Kurdyka--\L ojasiewicz property and $\tilde{f}(\x)=f(\x)$ over $\calP$, we obtain 
\begin{equation}\label{eq:KL-eq}
 \varphi^{\prime}(\tilde{f}(\y^k)-\tilde{f}(\x^*))\text{dist}(0,\partial^{F}\tilde{f}(\y^k))\ge 1.
\end{equation}
Because $\y^k\in\argmin_{\x}\{g(\x)-\langle\nabla h(\x^k),\x\rangle,\x\in\calP\},$
we get that
	$$\nabla h(\x^k)=\nabla g(\y^k)+\bs{u}_k, \bs{u}_k\in N_{\calP}(\y^k),$$
using the previous equation, we obtain $$\nabla h(\y^k)-\nabla h(\x^k)=\nabla h(\y^k)-\nabla g(\y^k)-\bs{u}_k\in-\partial^{F}\tilde{f}(\y^k).$$
Then it follows from \eqref{eq:xk-yk-lipschitz} that
\begin{equation}
\label{eq:dist-xk-yk}
\text{dist}(0,\partial^{F}\tilde{f}(\y^k))\le\|\nabla h(\y^k)-\nabla h(\x^k)\|\le L\|\x^k-\y^k\|.
\end{equation}

Combining \eqref{eq:KL-eq} and \eqref{eq:dist-xk-yk} with $\varphi(s)=Ms^{1-\theta}$, we can obtain for all $k\ge N$ that
\begin{equation}\label{eq:theta-xk-yk}
    (f(\y^k)-f_{opt})^{\theta}\le ML(1-\theta)\|\x^k-\y^k\|.
\end{equation}
For all $k\ge N$, it follows from $f(\y^k)-f_{opt}\ge f(\x^{k+1})-f_{opt}\ge0$, \eqref{eq:d-fxk} and \eqref{eq:theta-xk-yk} that
\begin{equation*}
    \begin{split}
    (f(\x^{k+1})-f_{opt})^{2\theta}\le&(f(\y^k)-f_{opt})^{2\theta}\\
    \le&(ML)^2(1-\theta)^2\|\x^k-\y^k\|^2\\
    \le&\frac{(ML)^2(1-\theta)^2}{\rho}(f(\x^k)-f(\x^{k+1}))\\
    =&\frac{(ML)^2(1-\theta)^2}{\rho}((f(\x^k)-f_{opt})-(f(\x^{k+1})-f_{opt})).
    \end{split}
\end{equation*}
Hence, setting $s_k=f(\x^k)-f_{opt}$ and $\tau=\frac{(ML)^2(1-\theta)^2}{\rho}>0$,  we can obtain 
\begin{equation*}
    s_{k+1}^{2\theta}\le\tau(s_k-s_{k+1})~~\text{for all }k\ge N.
\end{equation*}
Invoking Lemma \ref{lem:sk}, we can conclude (i), (ii) and (iii) immediately.
\qed\end{proof}

\section{An efficient method for subproblem \eqref{eq:subproblem}}\label{sec:FDPG}
In this section, we focus on how to solve subproblem \eqref{eq:subproblem} by introducing separable block structure into the function $g(\x)$ and leveraging the fast dual-based proximal gradient (FDGP) method \cite{beck2014fast}. FDGP amounts to computing the vector projection on the polyhedral set and evaluating a proximal mapping in each iteration. Particularly, evaluating the proximal mapping is equivalent to finding the unique root of a strictly convex univariate polynomial, which can be performed in parallel due to the separable block structure. The vector projection on the polyhedral set can also be computed efficiently by splitting the equality and inequality constraints.

Without loss of generality, let us consider the objective function of problem \eqref{eq:HD-PSDC} with a $\rho$-strongly convex DC decomposition as 
	\begin{equation}\label{eq:mrho}
		f(\x)=\underbrace{\|\bs{A}^{+}\x+\bs{b}^{+}\|_{d_f}^{d_f}+\frac{\rho}{2}\|\x\|^2}_{g(\x)}-\underbrace{\|\bs{A}^{-}\x+\bs{b}^{-}\|_{d_f}^{d_f}+\frac{\rho}{2}\|\x\|^2}_{h(\x)}.
	\end{equation}
The convex subproblem \eqref{eq:subproblem} reads 
	\begin{equation}\label{eq:convex_subpro}
		\min_{\x\in\calP}\big\{\|\bs{A}^{+}\x+\bs{b}^{+}\|_{d_f}^{d_f}+\frac{\rho}{2}\|\x\|^2-\langle\nabla h(\x^k),\x\rangle\big\},
	\end{equation}
where $\bs{A}^{+}=(\bs{a}_1^+,\bs{a}_2^+,\ldots,\bs{a}_{m_1}^+)^{\top}$ and $\bs{b}^+=(b_1^+,b_2^+,\ldots,b_{m_1}^+)^{\top}.$ 

Now, we will discuss the solution method for \eqref{eq:subproblem} by supposing $\chi_{\calP}$ to be prox-friendly or not respectively.
\paragraph{\textbf{Case 1: $\chi_{\calP}(\cdot)$ is prox-friendly.}} 
By incorporating $\chi_{\calP}(\x)$ into the objective function of \eqref{eq:convex_subpro}, we get 
	\begin{equation}\label{eq:pf}
		\min_{\x\in \R^n}\{\psi(\x)+\varphi( \bs{A}^+\x)\},\tag{P$_\text{f}$}
	\end{equation}
	where
	\begin{equation*}
		\begin{cases}
			\varphi( \bs{A}^+\x)=\sum_{i=1}^{m_1}\varphi_i(\langle\bs{a}_i^+,\x\rangle) \text{ with } \varphi_i(\langle\bs{a}_i^+,\x\rangle)=(\langle\bs{a}_i^+,\x\rangle+b_i)^{d_f},i=1,\ldots,m_1,\\
			\psi(\x)=\frac{\rho}{2}\|\x\|^2-\langle\nabla h(\x^k),\x\rangle+\chi_{\calP}(\x).
		\end{cases}
	\end{equation*} 
	\begin{algorithm}[ht!]
		\label{alg:FDPG-primal}
		\caption{FDPG method for \eqref{eq:pf}}
			 \KwIn{$\bs{w}^0=\z^0,s_0=1$, $L\ge \frac{\|\bs{A}^+\|^2}{\rho}$ and $\varepsilon^{\prime}>0.$}
	\For{$l=0,1,\ldots$}{
			\textbf{Step 1:} $\bs{u}^l=\proj\nolimits_{\calP}(\frac{1}{\rho}(\nabla h(\x^k)+(\bs{A}^+)^{\top}\bs{w}^l));$\\
			\textbf{Step 2:} $v_i^l=\prox\nolimits_{L\varphi_i}((\bs{a}_i^+)^{\top}\bs{u}^l-Lw_i^l), i=1,\ldots,m_1;$\\
			\textbf{Step 3:} $\z^{l+1}=\bs{w}^l-\frac{1}{L}(\bs{A}^+\bs{u}^l-\bs{v}^l);$\\
			\textbf{Step 4:} If $\|\z^{l+1}-\z^l\|/(1+\|\z^l\|)<\varepsilon^{\prime},$ \Return $\z^{l+1}$;\\
			\textbf{Step 5:} $s_{l+1}=\frac{1+\sqrt{1+4s_l^2}}{2};$\\
			\textbf{Step 6:} $\bs{w}^{l+1}=\z^{l+1}+(\frac{s_l-1}{s_{l+1}})(\z^{l+1}-\z^l).$
			}
\end{algorithm}
By introducing the linear constraint $\bs{A}^+\x-\bs{\zeta}=0$, the Lagrangian is
\begin{equation*}
	L(\x,\bs{\zeta};\z)=\psi(\x)+\varphi(\bs{\zeta})-\langle\bs{A}^+\x-\bs{\zeta},\z\rangle
\end{equation*}
with $\z=(z_1,z_2,\ldots,z_{m_1})^{\top}.$ Then the Lagrangian dual problem of \eqref{eq:pf} is given by 
\begin{equation}\label{eq:dual-problem}
		\min_{\z}\{\Psi(\z)+\Phi(\z)\},\tag{D$_\text{f}$}
\end{equation}
where $\Psi(\z)=\psi^*((\bs{A}^+)^{\top}\z),\Phi(\z)=\varphi^*(-\z).$
Applying the renowned FISTA \cite{beck2009fast} to \eqref{eq:dual-problem}, we get the FDPG method for \eqref{eq:pf} described in Algorithm \ref{alg:FDPG-primal}.

Some comments on Algorithm \ref{alg:FDPG-primal} are summarized below:
\begin{itemize}
    \item [(i)] In the settings of \textbf{Input}, we choose the initial point $\z^0=-\nabla\varphi(\bs{A}^+\x^k)$ by solving the KKT system with respect to $\bs{\zeta}$ and $\z$ at $\x^k$. 

\item[(ii)] In \textbf{Step 1},  the projection on $\calP$ (e.g., simplex or box) is easy to compute by assuming that $\chi_{\calP}(\cdot)$ is prox-friendly.

\item[(iii)] In \textbf{Step 2}, the proximal operator of all strictly convex functions $\varphi_i$ is computed by finding the unique real root of its derivative.  This can be achieved through a closed-form solution using the root formula when $d_f\le 5$, or through numerical computation using Newton's method when $d_f>5$. Notably, this step can be performed in parallel.

\end{itemize}
The convergence result of the sequence $\{\bs{u}^l\}$ is given bellow.
\begin{lemma}[See \cite{beck2014fast}]
	\label{lem:convergence-rate-xl}
	Suppose that $\{\bs{u}^l\}_{l\ge 0}$ is the sequence generated by Algorithm \ref{alg:FDPG-primal}. Then for the unique optimal solution $\x_{opt}$ of problem \eqref{eq:pf} and any optimal solution $\z_{opt}$ of problem \eqref{eq:dual-problem} and $l\ge 1$,
	$$\|\bs{u}^l-\x_{opt}\|\le2\sqrt{\frac{L}{\rho}}\frac{\|\z^0-\z_{opt}\|}{l+1}.$$	
\end{lemma}
Hence, by setting $L=\frac{\|\bs{A}^+\|^2}{\rho}$ in Algorithm \ref{alg:FDPG-primal}, we can simplify from Lemma \ref{lem:convergence-rate-xl} that 
\begin{equation}
	\|\bs{u}^l-\x_{opt}\|\le2 \frac{\|\bs{A}^+\|}{\rho}\frac{\|\z^0-\z_{opt}\|}{l+1}.
\end{equation}
\begin{lemma}\label{lem:A-c} Let $\bs{A}^{+}$ be defined in problem \eqref{eq:HD-PSDC} and $\bs{c}$ be the coefficient vector of its objective function $f(\x)$. Then
\begin{equation}\label{eq:a-C}
    \|\bs{A}^+\|\le \|\bs{L}_{n,d_f}\|\|\whV^{-1}(n,d_f)\|^{\frac{1}{d_f}}\|\bs{c}\|_{\infty}^{\frac{1}{d_f}},
\end{equation}
where $\bs{L}_{n,d_f}=(\balpha^1,\ldots,\balpha^{s_{n,d_f}})^{\top}$ and $d_f=2\lceil \frac{d}{2} \rceil$.
\end{lemma}
\begin{proof}
Following from Proposition \ref{cor:oddevenhomogenization} and \eqref{eq:HD-PSDC}, let us denote by $\bs{\lambda}=(\whV^{-1}(n,d_f))^{\top}\bs{c}$, then we have 
\begin{equation*}
  \begin{pmatrix}
  \bs{A}^{+} ~& \bs{b}^{+}\\
  \bs{A}^{-} ~& \bs{b}^{-}
  \end{pmatrix}=\bs{L}_{n,d_f}\begin{pmatrix}
  |\lambda_1|^{\frac{1}{d_f}}&&\\
  &\ddots&\\
  &&|\lambda_{s_{n,d}}|^{\frac{1}{d_f}}
  \end{pmatrix}.
\end{equation*}
Then according to the Cauchy interlacing theorem \cite{parlett1998symmetric}, we get
$$\|\bs{A}^{+}\|\le\|\begin{pmatrix}
  \bs{A}^{+} ~& \bs{b}^{+}\\
  \bs{A}^{-} ~& \bs{b}^{-}
  \end{pmatrix}\|\le\|\bs{L}_{n,d_f}\|(\|\bs{\lambda}\|_{\infty})^{\frac{1}{d_f}}.$$
Combining with $\|\bs{\lambda}\|_{\infty}\le \|\whV^{-1}(n,d_f)\|\|\bs{c}\|_{\infty}$, we get the desired conclusion.
\qed\end{proof}
Lamma \ref{lem:A-c} establishes the connection between $\|\bs{A}^{+}\|$ and the coefficients $\bs{c}$ of $f(\x)$. When $n$ and $d_f$ are given, then $\bs{L}_{n,d_f}$ and $\whV(n,d_f)$ are fixed, hence $$\|\bs{A}^+\| = O\left(\|\bs{c}\|_{\infty}^{\frac{1}{d_f}}\right).$$

\paragraph{\textbf{Case 2: $\chi_{\calP}(\cdot)$ is not prox-friendly.}} We split  $\calP$ into two parts: the inequalities $\calP_I$ and the equalities $\calP_E$ as
$$\calP_I=\{\x:\bs{A}\x\le\bs{b}\} \quad \text{ and }\quad \calP_E=\{\x:\bs{P}\x=\bs{q}\}$$
with $\bs{A}=(\bs{a}_1,\ldots,\bs{a}_m)^{\top},\bs{b}=(b_1,\ldots,b_m)^{\top}$ and $\bs{P}=(\bs{p}_1,\ldots,\bs{p}_r)^{\top},\bs{q}=(q_1,\ldots,q_r)^{\top}.$ Here, we assume that $\bs{P}$ has full row rank (otherwise, we use Gauss's elimination to get an equivalence equality constraint in full row rank). 
Thus, the indicator function $\chi_{\calP_I}$ can be rewritten as $$\chi_{\calP_I}(\x)=\chi_{\text{Box}[-\infty,\bs{b}]}(\bs{A}\x),$$
and problem \eqref{eq:convex_subpro} reads
\begin{equation}\label{eq:npf}
		\min_{\x}\{\psi(\x)+\varphi( \bs{A}^+\x)+\chi_{\text{Box}[-\infty,\bs{b}]}(\bs{A}\x)\},\tag{P$_\text{nf}$}
	\end{equation}
	where
	\begin{equation*}
		\begin{cases}
			\varphi( \bs{A}^+\x)=\sum_{i=1}^{m_1}\varphi_i(\langle\bs{a}_i^+,\x\rangle) \text{ with } \varphi_i(\langle\bs{a}_i^+,\x\rangle)=(\langle\bs{a}_i^+,\x\rangle+b_i)^{d_f},i=1,\ldots,m_1,\\
			\psi(\x)=\frac{\rho}{2}\|\x\|^2-\langle\nabla h(\x^k),\x\rangle+\chi_{\calP_E}(\x).
		\end{cases}
	\end{equation*} 
The FDPG can also be applied for solving \eqref{eq:npf} as in Case 1. The differences include:  
\begin{itemize}
    \item [(i)] In \textbf{Step 1}, we compute $w^l\to u^{l+\frac{1}{2}} \to u^{l}$ as 
\begin{equation*}
		\begin{cases}
		    \bs{u}^{l+\frac{1}{2}}=\frac{1}{\rho}\left[\nabla h(\x^k)+\begin{pmatrix}
                    \bs{A}^+\\
                    \bs{A}
                    \end{pmatrix}^{\top}\bs{w}^l\right];\\
            \bs{u}^l=\bs{u}^{l+\frac{1}{2}}-\bs{P}^{\top}(\bs{P}\bs{P}^{\top})^{-1}(\bs{P}\bs{u}^{l+\frac{1}{2}}-\bs{q}).
		\end{cases}
\end{equation*}
\item [(ii)] In \textbf{Step 2}, there are some extra proximal operators to compute as
$$v_i^l=\begin{cases}
    \prox\nolimits_{L\varphi_i}((\bs{a}_i^+)^{\top}\bs{u}^l-Lw_i^l),& i=1,\ldots,m_1;\\
   \min\{\bs{a}_i^{\top}\bs{u}^l-Lw_i^l,~b_i\},& i=m_1+1,\ldots,m_1+m.  
\end{cases}$$

\item [(iii)] In \textbf{Step 3}, the $\z^{l+1}$ is updated by:
$$\z^{l+1}=\bs{w}^l-\frac{1}{L}\left[\begin{pmatrix}
                    \bs{A}^+\\
                    \bs{A}
        \end{pmatrix}\bs{u}^l-\bs{v}^l\right].$$
\end{itemize}
Note that similar method can also be viable for solving the convex subproblem \eqref{eq:subproblem} for \eqref{eq:T-PSDC}.
\section{Numerical experiment}\label{sec:numerical-experiment}
In this section, we test BDCA$_\text{e}$ against the classical DCA and BDCA \cite{niu2019higher} applied to our new PSDC decompositions, UDCA \cite{Niu2011} and UBDCA \cite{niu2019higher} (with projective DC decomposition) as well as \texttt{FILTERSD} \cite{filtersdFletcher} and \texttt{FMINCON} for solving the higher-order moment portfolio optimization model and the polynomial optimization problem with box constraint. All experiments are performed on $\text{MATLAB}$ R2021a with a laptop equipped with Intel Core i5-1035G1 CPU 1.19GHz and 8GB of RAM. According to Lemma \ref{lem:psrep}, the gradient computation of any form $\varphi(\x)\in\bbH_d[\x]$ is given by
\begin{equation}
 \label{eq:grad}
\nabla \varphi(\x)=d\sum\nolimits_{\balpha\in \calI_{n,d}}\lambda_{\balpha}\langle\balpha, \x\rangle^{d-1}\balpha.
\end{equation}
\subsection{The Mean-Variance-Skewness-Kurtosis portfolio optimization problem}
Let $\bs{r}$ = $(r_1, r_2,\ldots,r_n)^{\top}$ be a random return vector of $n$ risky assets, $\x$ be the $n$-dimensional decision vector, where $x_i$ is the percentage invested in the $i$-th risky asset, and $\E$ be the expectation operator. The moments of the return are defined as the following: the mean of the return denoted by $\bs{\mu}$ $\in$ $\R^n$ and $\forall i \in \{1,\ldots,n\}$,$\mu_i = \E(r_i);$ the variance of the return denoted by $\bs{V}$ $\in$ $\R^{n^2}$ and $\forall (i, j)\in \{1,\ldots,n\}^2$, $V_{ij} = \E\left([r_i - \E(r_i)][r_j - \E(r_j)]\right);$ the skewness of the return denoted by $\bs{S}$ $\in$ $\R^{n^3}$ and $\forall (i, j, k)\in \{1,\ldots,n\}^3$, $S_{ijk} = \E\left([r_i -\E(r_i)][r_j - \E(r_j)][r_k - \E(r_k)]\right);$ the kurtosis of the return denoted by $\bs{K}$ $\in\R^{n^4}$ and $\forall (i, j, k, l)\in \{1,\ldots,n\}^4$, $K_{ijkl} = \E([r_i - \E(r_i)][r_j - \E(r_j)][r_k - \E(r_k)][r_l - \E(r_l)]).$
The mean-variance-skewness-kurtosis portfolio model can be studied as a weighted nonconvex polynomial minimization problem (see, e.g., \cite{niu2019higher,Niu2011}): \begin{equation}\label{eq:MVSK}
    	\min_{\x} \{f(\x) : \bs{e}^{\top}\x=1,~\x\ge\bs{0},~\x\in\R^n\},
\tag{MVSK}
\end{equation}
where $$f(\x) := -\omega_1\sum_i\mu_i x_i+\omega_2\sum_{i,j} V_{ij}x_ix_j-\omega_3\sum_{i,j,k}S_{ijk}x_ix_jx_k+\omega_4\sum_{i,j,k,l}K_{ijkl}x_ix_jx_kx_l$$ with the parameter $\bs{\omega}=(\omega_1,\omega_2,\omega_3,\omega_4)\in\R^4_{+}$ being the investor's preference. Note that for some large $n$ (e.g., $n=100$), the total number of monomials in kurtosis can reach up to millions. However, due to the symmetry of $\bs{V}$ (resp. $\bs{S}$ and $\bs{K}$), the computation complexity on $\bs{V}$ (resp. $\bs{S}$ and $\bs{K}$) can be reduced from $n^2$ (resp. $n^3$ and $n^4$) to $\binom{n+1}{2}$ (resp. $\binom{n+2}{3}$ and $\binom{n+3}{4}$), see e.g., \cite{Maringer2009GlobalOO,Niu2011}. 

\textbf{Data generation:} We test on the dataset collected from monthly records from 1995 to 2015 of 48 industry-sector portfolios of the U.S. stock market with the number of assets $n\in\{11,16,21,26,31,36,41,46\}$. For each $n$, we test problems with three different investor's preferences, including risk-seeking $\bs{\omega}=(10,1,10,1)$, risk-aversing $\bs{\omega}= (1, 10, 1, 10)$ and risk-neutral $\bs{\omega}=(10,10,10,10)$.

\textbf{Setup:} The initial points are generated using MATLAB function \verb|rand|: we first set \verb|x0=rand(n,1)|, then compute \verb|x0=x0/sum(x0)| to get a feasible initial point $\x^0\in\calP$. The gradients required in DCA, BDCA and BDCA$_\text{e}$ are computed by Equation \eqref{eq:grad} and the gradients required in UDCA, UBDCA, \texttt{FILTERSD} and \texttt{FMINCON} for model \eqref{eq:MVSK} are computed by the next formula (see e.g., \cite{ATHAYDE2003243,Niu2011}):
$$\nabla f(\x)=-\omega_1\bs{\mu}+2\omega_2\bs{V}\x-3\omega_3\bs{S}(\x\otimes\x)+4\omega_4\bs{K}(\x\otimes\x\otimes\x).$$ We terminate DCA, BDCA and BDCA$_\text{e}$ with PSDC decomposition when 
\begin{equation*}
\begin{split}
     &\|\y^k-\x^k\|/(1+\|\x^k\|)\le 10^{-3}.\\
\end{split}
\end{equation*}
Algorithm \ref{alg:FDPG-primal} for \eqref{eq:subproblem} is terminated if $$\|\z^{l+1}-\z^l\|/(1+\|\z^l\|)\le 10^{-3}.$$
UDCA and UBDCA are terminated if 
\begin{equation*}
     \|\y^k-\x^k\|/(1+\|\x^k\|)\le 10^{-5}.
\end{equation*}
\texttt{FILTERSD} and \texttt{FMINCON} are terminated with default parameters.\\
The initial step size for Armijo line search is computed as in \cite{niu2019higher} by $$t=\min\left\{\min_{i\in\mathcal{I}(\bs{d}^k)}\{\frac{b_i-\langle\bs{a}_i,\y^k\rangle}{\langle\bs{a}_i,\bs{d}^k\rangle}\},\frac{\sqrt{2}}{\|\bs{d}^k\|}\right\}.$$ 
We reduce $t$ by $t=\beta t$ with $\beta=0.8$ until  
$$f(\x^{k+1})-f(\y^k)\le-\sigma t^2\|\bs{d}^k\|^2$$
is verified, where $\sigma=10^{-3}$ is used. 

\paragraph{\textbf{Numerical results}}:
In Table \ref{tab:tablet}--\ref{tab:tableo}, The objective values (\verb|obj|), the CPU time (\verb|time| in seconds), the number of iterations (\verb|iter|), the average CPU time and the average number of iterations are reported by comparing BDCA$_\text{e}$, DCA and BDCA applied to \eqref{eq:T-PSDC} and \eqref{eq:HD-PSDC}, UDCA and UBDCA with projective DC decomposition, and the solvers \texttt{FILTERSD} and \texttt{FMINCON} on problem \eqref{eq:MVSK}.  

We observe that BDCA$_\text{e}$ is as fast as BDCA for solving \eqref{eq:HD-PSDC}, and both are the fastest methods among the others. Regarding the objective values, \texttt{FILTERSD} often provides the best ones. BDCA$_\text{e}$ for both \eqref{eq:T-PSDC} and \eqref{eq:HD-PSDC} shares the same objective values as \texttt{FILTERSD} in most of the cases (15/24) with the difference of order $O(10^{-5})$. Moreover, BDCA$_\text{e}$ and BDCA for both \eqref{eq:T-PSDC} and \eqref{eq:HD-PSDC} almost always require fewer iterations and result in better objective values than UDCA and UBDCA. This implies that the quality of the proposed two PSDC decompositions is better than the projective DC decomposition. Concerning the CPU time, BDCA$_\text{e}$ for \eqref{eq:HD-PSDC} is almost as fast as BDCA \eqref{eq:HD-PSDC} and \texttt{FMINCON}, and at least $2$ times faster than DCA and UDCA. We conclude that BDCA$_\text{e}$ and BDCA for \eqref{eq:T-PSDC} and \eqref{eq:HD-PSDC} are promising approaches for solving \eqref{eq:MVSK}.

\begin{sidewaystable}[htbp]
\centering
\resizebox{0.8\linewidth}{!}{
 \begin{tabular}{c|c|*{3}{c}|*{3}{c}|*{3}{c}}
   \toprule
 \multirow{3}*{}&\multirow{3}*{$n$}&\multicolumn{9}{c}{Algorithms for solving \eqref{eq:T-PSDC}}\\
  \cline{3-11}&&\multicolumn{3}{c|}{DCA }&\multicolumn{3}{c|}{BDCA }& \multicolumn{3}{c}{BDCA$_\text{e}$ }\\
    \cline{3-11}&&iter&time&obj&iter&time&obj&iter&time&obj\\
     \midrule
1&11&113&0.04&\textbf{-8.705e-02}&19&0.02&\textbf{-8.705e-02}&10&0.01&\textbf{-8.705e-02}\\
2&11&80&0.05&7.378e-03&22&0.03&7.354e-03&12&0.02&7.348e-03\\
3&11&90&0.05&-9.608e-02&15&0.02&-9.608e-02&11&0.02&-9.609e-02\\
4&16&24&0.04&\textbf{-1.646e-01}&15&0.04&\textbf{-1.646e-01}&12&0.04&\textbf{-1.646e-01}\\
5&16&89&0.13&8.466e-03&23&0.05&8.391e-03&13&0.05&8.391e-03\\
6&16&63&0.07&-1.075e-01&20&0.04&\textbf{-1.076e-01}&15&0.03&\textbf{-1.076e-01}\\
7&21&119&0.24&\textbf{-1.302e-01}&17&0.05&\textbf{-1.302e-01}&12&0.04&\textbf{-1.302e-01}\\
8&21&58&0.14&5.191e-03&23&0.07&5.194e-03&16&0.10&5.192e-03\\
9&21&78&0.17&\textbf{-1.079e-01}&19&0.09&\textbf{-1.079e-01}&14&0.07&\textbf{-1.079e-01}\\
10&26&118&0.49&\textbf{-1.047e-01}&23&0.13&\textbf{-1.047e-01}&16&0.14&\textbf{-1.047e-01}\\
11&26&85&0.40&2.689e-03&26&0.17&2.687e-03&16&0.18&2.649e-03\\
12&26&65&0.28&-9.899e-02&13&0.08&-9.898e-02&10&0.07&\textbf{-9.902e-02}\\
13&31&56&0.39&\textbf{-1.649e-01}&17&0.16&\textbf{-1.649e-01}&12&0.18&\textbf{-1.649e-01}\\
14&31&94&0.76&1.139e-03&22&0.24&1.051e-03&17&0.32&1.057e-03\\
15&31&66&0.48&-1.142e-01&15&0.14&\textbf{-1.143e-01}&11&0.15&\textbf{-1.143e-01}\\
16&36&59&0.74&\textbf{-1.649e-01}&17&0.29&\textbf{-1.649e-01}&14&0.27&\textbf{-1.649e-01}\\
17&36&93&1.31&8.497e-04&23&0.44&7.928e-04&16&0.55&7.838e-04\\
18&36&55&0.72&\textbf{-1.057e-01}&20&0.40&\textbf{-1.057e-01}&15&0.31&\textbf{-1.057e-01}\\
19&41&58&1.11&\textbf{-1.650e-01}&17&0.43&\textbf{-1.650e-01}&12&0.41&\textbf{-1.650e-01}\\
20&41&77&1.67&1.409e-03&29&0.84&1.412e-03&22&0.95&1.354e-03\\
21&41&72&1.52&\textbf{-1.145e-01}&21&0.72&\textbf{-1.145e-01}&14&0.44&\textbf{-1.145e-01}\\
22&46&58&2.14&\textbf{-1.650e-01}&18&0.79&\textbf{-1.650e-01}&14&0.74&\textbf{-1.650e-01}\\
23&46&96&3.96&7.829e-04&23&1.20&7.303e-04&19&1.31&7.138e-04\\
24&46&72&2.74&-1.145e-01&25&1.40&\textbf{-1.146e-01}&14&0.98&\textbf{-1.146e-01}\\
\hline
\multicolumn{2}{c|}{average}&77&0.82&&20&0.33&&14&0.31&\\
\bottomrule
\end{tabular}}
\caption{Numerical results of DCA, BDCA an BDCA$_{\text{e}}$ with with $\rho=1$ for problem \eqref{eq:T-PSDC}; Bold values for best results in objective function.}\label{tab:tablet}
\end{sidewaystable}

\begin{sidewaystable}[htbp]
 \centering
 \resizebox{0.8\linewidth}{!}{
 \begin{tabular}{c|c|*{3}{c}|*{3}{c}|*{3}{c}}
   \toprule
 \multirow{3}*{}&\multirow{3}*{$n$}&\multicolumn{9}{c}{Algorithms for solving \eqref{eq:HD-PSDC}}\\
  \cline{3-11}&&\multicolumn{3}{c|}{DCA }&\multicolumn{3}{c|}{BDCA }& \multicolumn{3}{c}{BDCA$_\text{e}$ }\\
    \cline{3-11}&&iter&time&obj&iter&time&obj&iter&time&obj\\
     \midrule
1&11&113&0.03&\textbf{-8.705e-02}&18&0.01&\textbf{-8.705e-02}&9&0.01&\textbf{-8.705e-02}\\
2&11&74&0.02&7.378e-03&23&0.01&7.355e-03&13&0.01&7.348e-03\\
3&11&86&0.02&-9.608e-02&13&0.00&-9.608e-02&9&0.01&-9.608e-02\\
4&16&22&0.02&\textbf{-1.646e-01}&15&0.02&\textbf{-1.646e-01}&13&0.02&\textbf{-1.646e-01}\\
5&16&89&0.05&8.466e-03&23&0.03&8.391e-03&13&0.03&8.391e-03\\
6&16&60&0.03&-1.075e-01&19&0.03&\textbf{-1.076e-01}&14&0.02&\textbf{-1.076e-01}\\
7&21&119&0.19&\textbf{-1.302e-01}&17&0.04&\textbf{-1.302e-01}&12&0.04&\textbf{-1.302e-01}\\
8&21&58&0.09&5.191e-03&22&0.05&5.194e-03&15&0.06&5.190e-03\\
9&21&78&0.11&\textbf{-1.079e-01}&19&0.05&\textbf{-1.079e-01}&14&0.04&\textbf{-1.079e-01}\\
10&26&118&0.40&\textbf{-1.047e-01}&23&0.11&\textbf{-1.047e-01}&16&0.11&\textbf{-1.047e-01}\\
11&26&85&0.31&2.689e-03&26&0.14&2.687e-03&19&0.16&2.647e-03\\
12&26&65&0.23&-9.899e-02&13&0.07&-9.898e-02&10&0.05&\textbf{-9.902e-02}\\
13&31&56&0.35&\textbf{-1.649e-01}&17&0.15&\textbf{-1.649e-01}&12&0.16&\textbf{-1.649e-01}\\
14&31&94&0.65&1.139e-03&22&0.20&1.051e-03&17&0.28&1.057e-03\\
15&31&66&0.42&-1.142e-01&15&0.12&\textbf{-1.143e-01}&11&0.13&\textbf{-1.143e-01}\\
16&36&59&0.66&\textbf{-1.649e-01}&17&0.24&\textbf{-1.649e-01}&14&0.23&\textbf{-1.649e-01}\\
17&36&93&1.16&8.497e-04&23&0.38&7.928e-04&16&0.44&7.838e-04\\
18&36&55&0.62&\textbf{-1.057e-01}&20&0.35&\textbf{-1.057e-01}&15&0.27&\textbf{-1.057e-01}\\
19&41&58&1.04&\textbf{-1.650e-01}&17&0.40&\textbf{-1.650e-01}&12&0.36&\textbf{-1.650e-01}\\
20&41&77&1.49&1.409e-03&29&0.77&1.412e-03&22&0.81&1.354e-03\\
21&41&72&1.32&\textbf{-1.145e-01}&21&0.64&\textbf{-1.145e-01}&13&0.38&\textbf{-1.145e-01}\\
22&46&58&1.96&\textbf{-1.650e-01}&18&0.74&\textbf{-1.650e-01}&14&0.66&\textbf{-1.650e-01}\\
23&46&96&3.56&7.829e-04&23&1.02&7.303e-04&20&1.15&7.138e-04\\
24&46&72&2.50&-1.145e-01&25&1.29&\textbf{-1.146e-01}&16&0.95&\textbf{-1.146e-01}\\
\hline
\multicolumn{2}{c|}{average}&76&0.72&&20&0.29&&14&0.27&\\
\bottomrule
\end{tabular}}
\caption{Numerical results of DCA, BDCA and BDCA$_{\text{e}}$ with $\rho=1$ for problem \eqref{eq:HD-PSDC}; Bold values for best results in objective function.}\label{tab:tableh}
\end{sidewaystable}

\begin{sidewaystable}[htbp]
 \centering
\resizebox{0.9\linewidth}{!}{
 \begin{tabular}{c|c|*{3}{c}|*{3}{c}|*{2}{c}|*{2}{c}}
 \toprule
 \multirow{2}*{}&\multirow{2}*{$n$}&\multicolumn{3}{c|}{UDCA}&\multicolumn{3}{c|}{UBDCA}&\multicolumn{2}{c|}{\texttt{FILTERSD}}&\multicolumn{2}{c}{\texttt{FMINCON}}\\
\cline{3-12}&&iter&time&obj&iter&time&obj&time&obj&time&obj\\
\midrule
1&11&404&0.01&\textbf{-8.705e-02}&25&0.00&\textbf{-8.705e-02}&0.00&\textbf{-8.705e-02}&0.01&\textbf{-8.705e-02}\\
2&11&1900&0.13&7.345e-03&51&0.02&\textbf{7.343e-03}&0.01&\textbf{7.343e-03}&0.02&7.346e-03\\
3&11&1698&0.06&\textbf{-9.611e-02}&65&0.01&\textbf{-9.611e-02}&0.00&\textbf{-9.611e-02}&0.01&\textbf{-9.611e-02}\\
4&16&84&0.00&\textbf{-1.646e-01}&26&0.00&\textbf{-1.646e-01}&0.01&\textbf{-1.646e-01}&0.03&\textbf{-1.646e-01}\\
5&16&6983&0.42&8.410e-03&86&0.05&8.381e-03&0.03&\textbf{8.377e-03}&0.02&8.381e-03\\
6&16&4065&0.21&\textbf{-1.076e-01}&85&0.04&\textbf{-1.076e-01}&0.02&\textbf{-1.076e-01}&0.02&\textbf{-1.076e-01}\\
7&21&1653&0.22&\textbf{-1.302e-01}&50&0.03&\textbf{-1.302e-01}&0.04&\textbf{-1.302e-01}&0.04&\textbf{-1.302e-01}\\
8&21&5609&0.85&5.163e-03&69&0.03&5.136e-03&0.03&\textbf{5.108e-03}&0.03&5.113e-03\\
9&21&8290&1.01&-1.078e-01&82&0.04&\textbf{-1.079e-01}&0.06&\textbf{-1.079e-01}&0.07&\textbf{-1.079e-01}\\
10&26&3814&0.89&\textbf{-1.047e-01}&75&0.04&\textbf{-1.047e-01}&0.05&\textbf{-1.047e-01}&0.04&\textbf{-1.047e-01}\\
11&26&10382&2.64&3.018e-03&91&0.10&2.768e-03&0.07&\textbf{2.610e-03}&0.07&2.617e-03\\
12&26&9422&2.05&-9.884e-02&79&0.04&-9.895e-02&0.06&\textbf{-9.902e-02}&0.08&-9.901e-02\\
13&31&2114&1.00&\textbf{-1.649e-01}&75&0.07&\textbf{-1.649e-01}&0.16&\textbf{-1.649e-01}&0.09&\textbf{-1.649e-01}\\
14&31&12294&6.18&1.771e-03&104&0.13&1.595e-03&0.11&\textbf{1.042e-03}&0.12&1.050e-03\\
15&31&8182&3.77&-1.140e-01&113&0.15&-1.140e-01&0.10&\textbf{-1.143e-01}&0.14&\textbf{-1.143e-01}\\
16&36&4129&3.95&\textbf{-1.649e-01}&100&0.22&\textbf{-1.649e-01}&0.16&\textbf{-1.649e-01}&0.22&\textbf{-1.649e-01}\\
17&36&9511&9.32&1.604e-03&114&0.28&1.589e-03&0.29&\textbf{7.721e-04}&0.27&7.816e-04\\
18&36&10888&10.38&-1.048e-01&124&0.48&-1.053e-01&0.48&\textbf{-1.057e-01}&0.37&\textbf{-1.057e-01}\\
19&41&4934&8.00&\textbf{-1.650e-01}&114&0.41&\textbf{-1.650e-01}&0.95&\textbf{-1.650e-01}&0.31&\textbf{-1.650e-01}\\
20&41&13390&22.20&3.186e-03&125&0.54&3.166e-03&0.79&\textbf{1.349e-03}&1.07&1.360e-03\\
21&41&12452&20.06&-1.134e-01&148&2.07&-1.135e-01&0.56&\textbf{-1.145e-01}&0.72&\textbf{-1.145e-01}\\
22&46&7394&18.74&\textbf{-1.650e-01}&132&0.77&\textbf{-1.650e-01}&1.18&\textbf{-1.650e-01}&0.78&\textbf{-1.650e-01}\\
23&46&9783&25.16&2.354e-03&137&0.98&2.302e-03&1.37&\textbf{7.006e-04}&1.58&7.126e-04\\
24&46&15278&38.82&-1.122e-01&157&1.97&-1.125e-01&1.88&\textbf{-1.146e-01}&1.25&\textbf{-1.146e-01}\\
\hline
\multicolumn{2}{c|}{average}&6861&7.35&&93&0.35&&0.35&&0.31&\\
\bottomrule
\end{tabular}}
\caption{Numerical results of UDCA, UBDCA, \texttt{FILTERSD} and \texttt{FMINCON} for \eqref{eq:MVSK}; Bold values for best results in objective function.}\label{tab:tableo}
\end{sidewaystable} 
Fig. \ref{fig:solver_log_cpu_time_new} illustrates the $\log$ average CPU time along the number of assets $n$ from the results in Table \ref{tab:tablet}--\ref{tab:tableo}. We observe that the DCA, BDCA, and BDCA$_\text{e}$ algorithms, when applied to \eqref{eq:HD-PSDC}, exhibit faster performance compared to their counterparts applied to \eqref{eq:T-PSDC}. Furthermore, UDCA has the worst performance over all methods. \texttt{FILTERSD} is the fastest method when $n\leq 40$. As $n>40$, the best performance is given by BDCA$_\text{e}$ for \eqref{eq:HD-PSDC}, which demonstrates that BDCA$_\text{e}$ should be a promising approach on large-scale settings.
\begin{figure}[ht!]
	\centering
	\includegraphics[width=10cm,height=6cm]{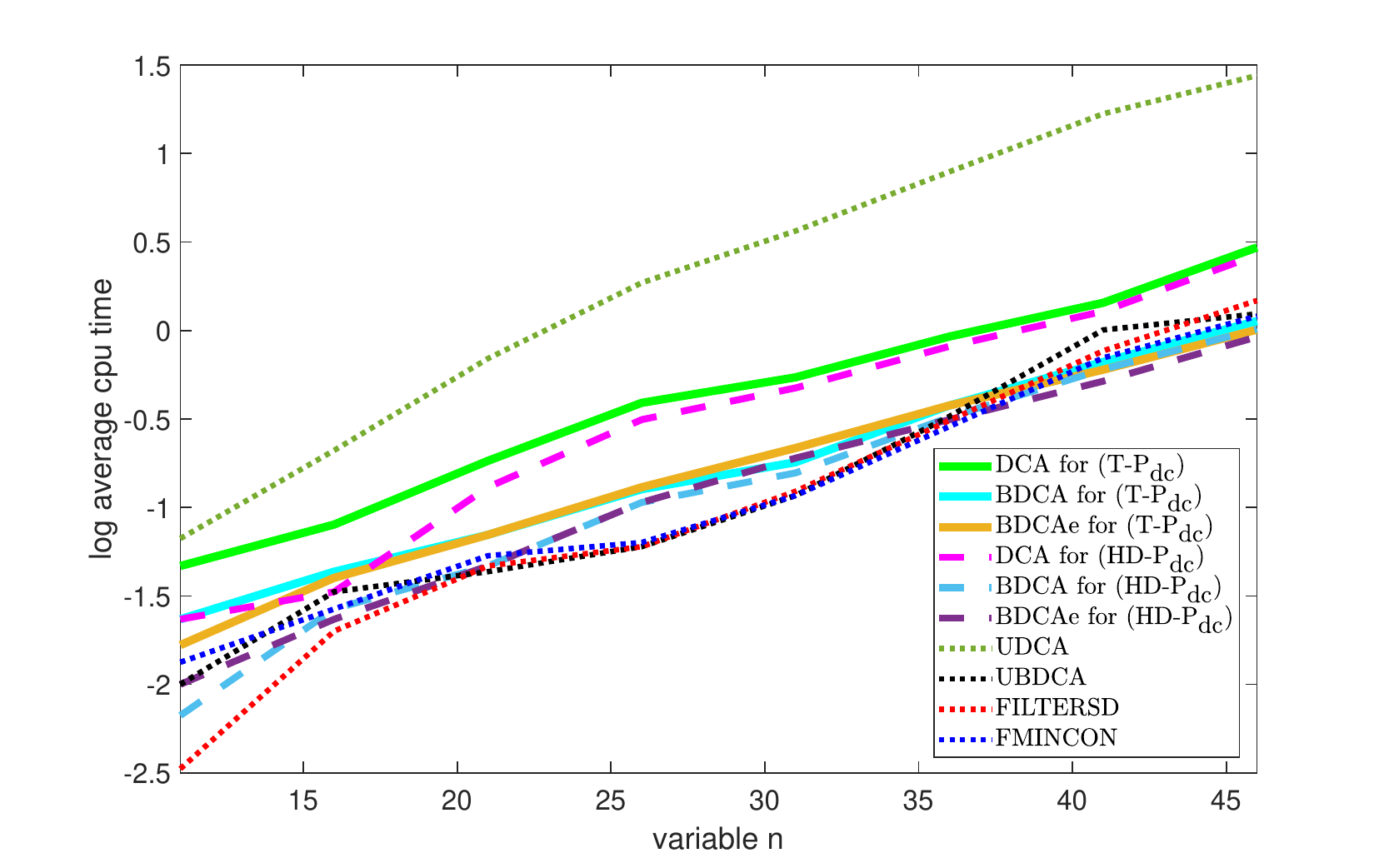}
	\caption{Log average CPU time along $n$ for results of DCA, BDCA and BDCA$_\text{e}$ for \eqref{eq:T-PSDC} and \eqref{eq:HD-PSDC} as well as \texttt{FILTERSD} and \texttt{FMINCON} reported in Tables \ref{tab:tablet}--\ref{tab:tableo}.}
	\label{fig:solver_log_cpu_time_new}
\end{figure}
\subsection{Polynomial optimization problem with box constraints}\label{sec:box}
We consider the box-constrained polynomial optimization problem(see \cite{anjos2011handbook,de2017improved}):
\begin{equation}\label{eq:Box}
    	\min_{\x} \{f(\x): \bs{l}\le\x\le \bs{u},~\x\in\R^n\}
\tag{P$_\text{box}$}
\end{equation}
where $f(\x)$ is a multivariate polynomial on $\R_d[\x]$ and $\bs{l},\bs{u}$ are lower and upper bound of $\x$. We compare BDCA$_\text{e}$ against the classical DCA and BDCA \cite{niu2019higher} applied to T-PSDC and HD-PSDC decompositions, respectively, as well as \texttt{FILTERSD} \cite{filtersdFletcher} and \texttt{FMINCON} for solving  \eqref{eq:Box}.

\textbf{Data generation:} The objective functions $f(\x)$ in our experiments are randomly generated in $\R_d[\x]$, whose coefficients are uniformly distributed in the interval $[-1,1]$, with a specified density parameter $den\in (0,1]$. Additionally, we set $[\bs{l},\bs{u}]=[-\bs{e},\bs{e}]$.

\textbf{Setup:} The initial points are randomly generated by MATLAB function \verb|rand|. We start by setting \verb|x0 = rand(n,1)|, and then transform it with \verb|x0 = -1 + 2 * x0| to obtain a feasible initial point $\x^0 \in \calP$.

The gradients required in DCA, BDCA and BDCA$_\text{e}$ are computed by Equation \eqref{eq:grad} and the gradients required in \texttt{FILTERSD} and \texttt{FMINCON} are approximated by $$\frac{\partial f}{\partial x_i}=(f(\x+\delta\bs{e}_i)-f(\x-\delta\bs{e}_i))/(2\delta),i=1,\ldots,n,$$
where $\delta=10^{-3}$. We terminate DCA, BDCA and BDCA$_\text{e}$ when 
\begin{equation*}
\begin{split}
     &\|\y^k-\x^k\|/(1+\|\x^k\|)\le 5\times 10^{-4}.\\
\end{split}
\end{equation*}
Algorithm \ref{alg:FDPG-primal} for \eqref{eq:subproblem} is terminated if $$\|\z^{l+1}-\z^l\|/(1+\|\z^l\|)\le 5\times10^{-5}.$$
The setting for the Armijo line search is the same as in \eqref{eq:MVSK}.  

\paragraph{\textbf{Numerical results:}} For each triple $(n,d,den)$ with $n\in\{10,20,30,40,50\},d\in\{3,4\}$ and $den\in\{0.1,0.4,0.7,1\}$, we generate $10$ independent trials and compare five methods (namely, DCA, BDCA, BDCA$_{e}$ with HD-PSDC decompositions, \texttt{FILTERSD} and \texttt{FMINCON}) in terms of the average CPU time (in seconds) over the $10$ trials.
\begin{figure}[ht!]
 	\centering
 	\subfigure{
 	\includegraphics[width=0.48\linewidth]{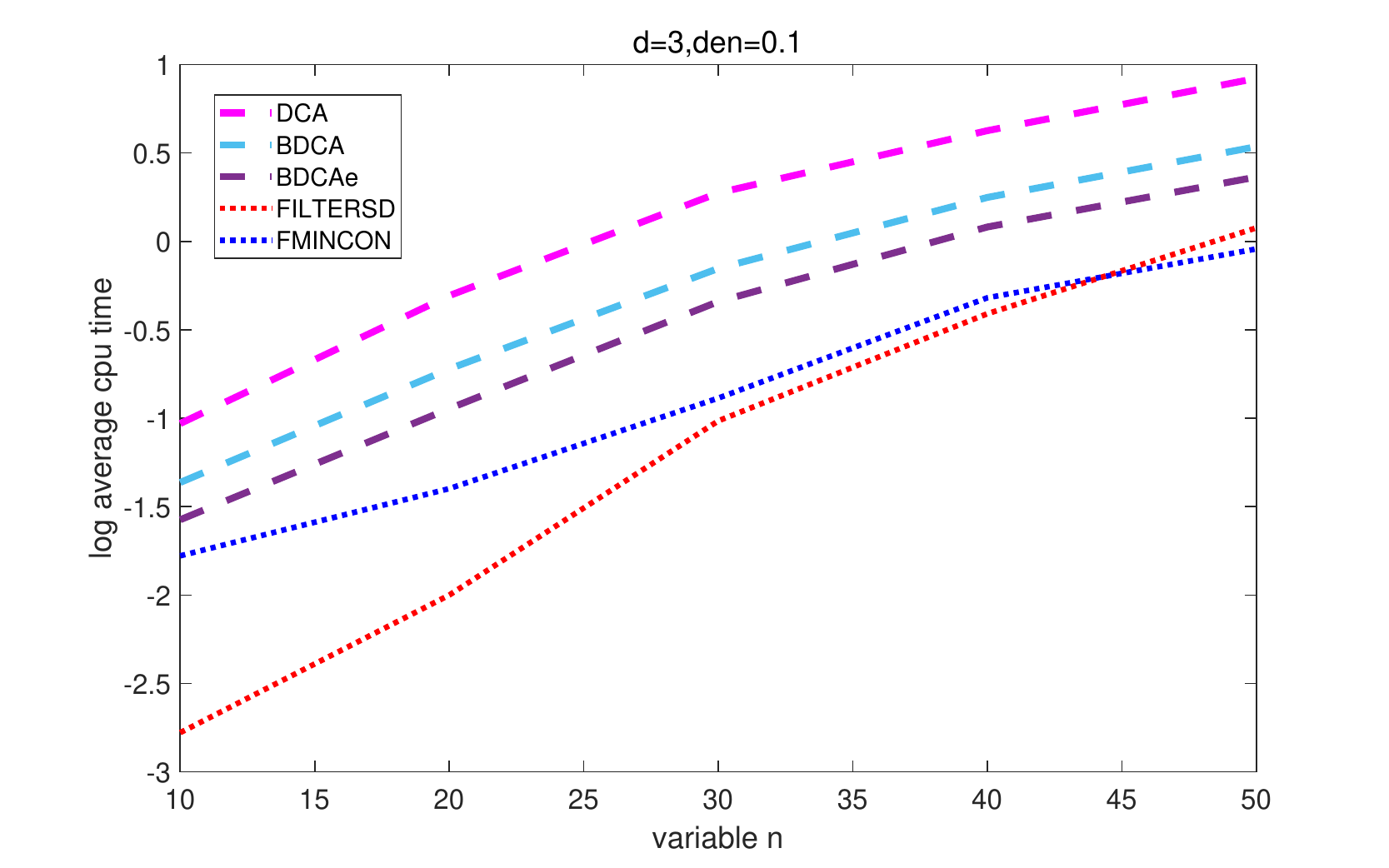}
 	\label{fig:d301}}
 	\subfigure{
 	\includegraphics[width=0.48\linewidth]{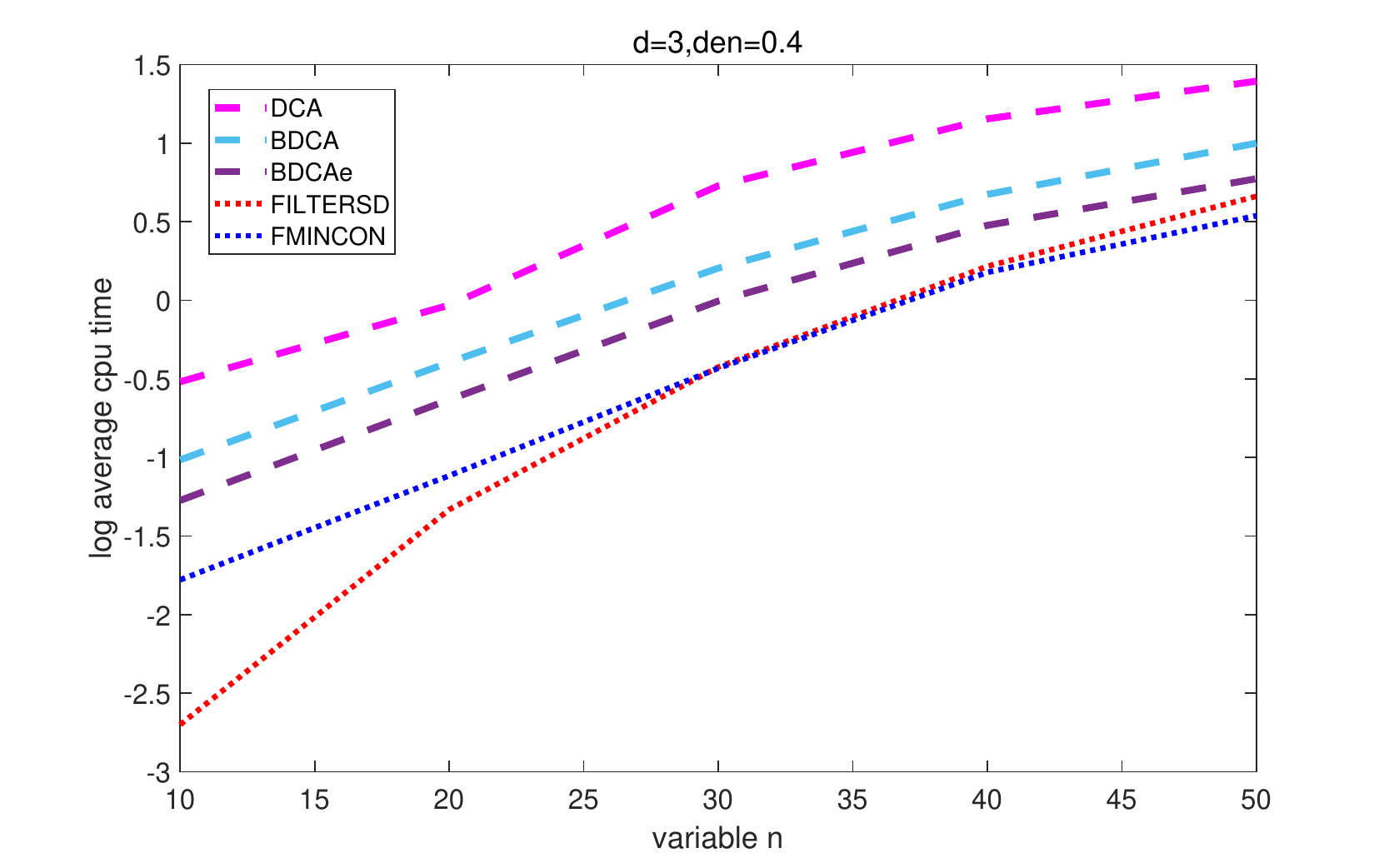}\label{fig:d304}}
 	\subfigure{
 	\includegraphics[width=0.48\linewidth]{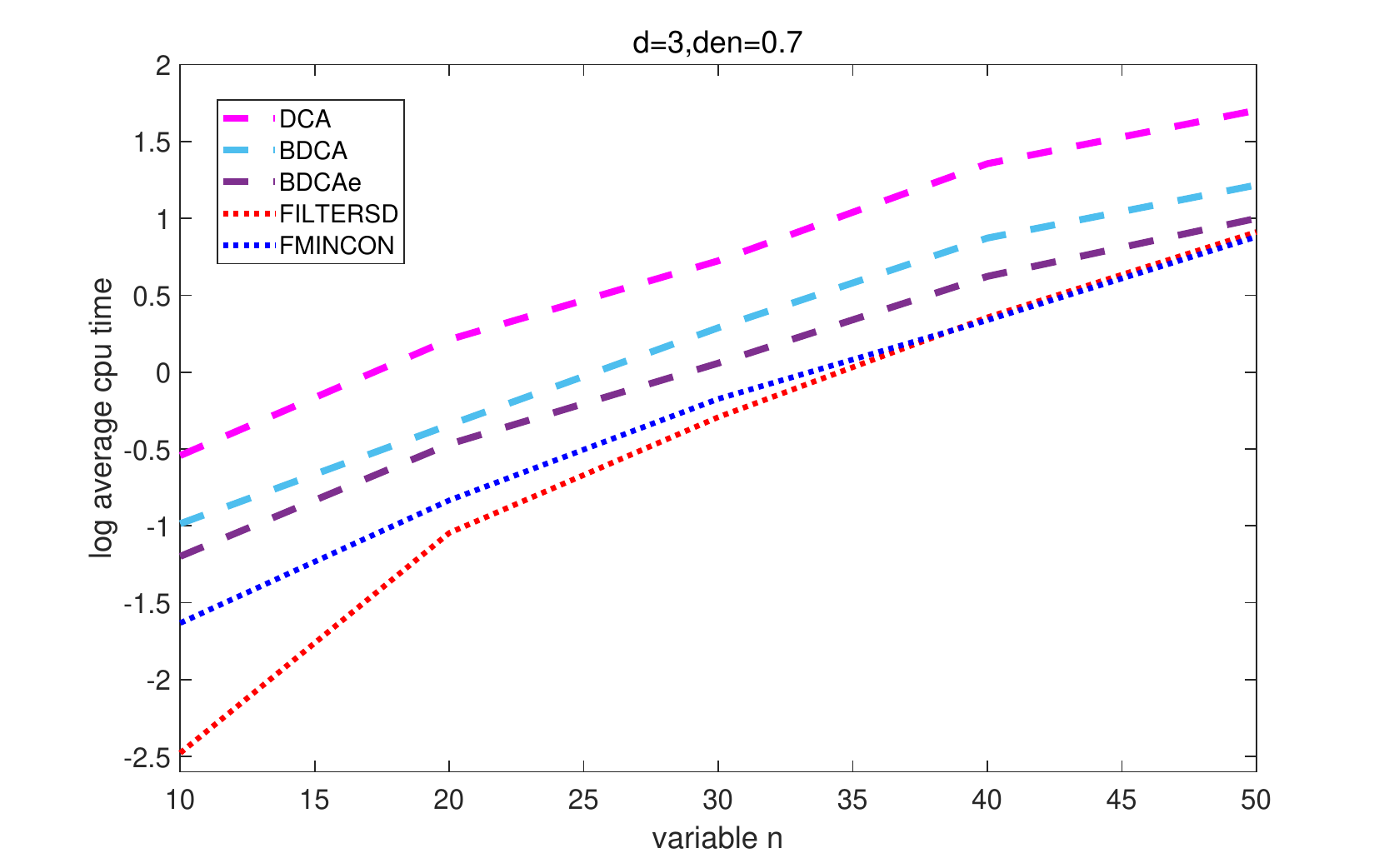}
 	\label{fig:d307}}
 	\subfigure{
 	\includegraphics[width=0.48\linewidth]{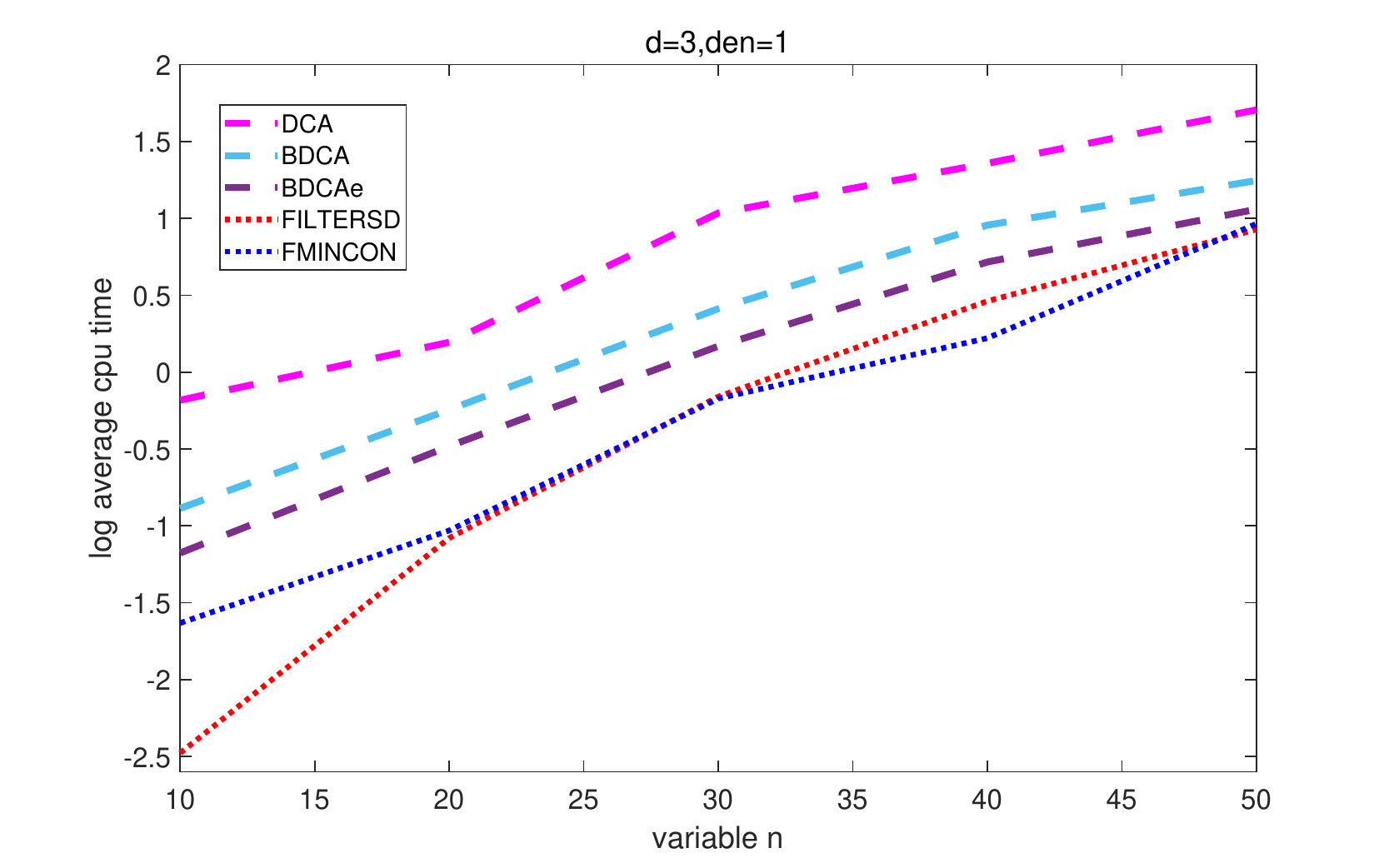}\label{fig:d31}}
 	\caption{Comparison of DCA, BDCA and BDCA$_\text{e}$ applied to HD-PSDC decomposition with $\rho=1$, as well as \texttt{FILTERSD} and \texttt{FMINCON} for solving \eqref{eq:Box} with $n\in\{10,20,30,40,50\},d=3, den\in\{0.1,0.4,0.7,1\}$.}
 	\label{fig:d-3}
\end{figure} 
\begin{figure}[ht!]
 	\centering
 	\subfigure{
 	\includegraphics[width=0.48\linewidth]{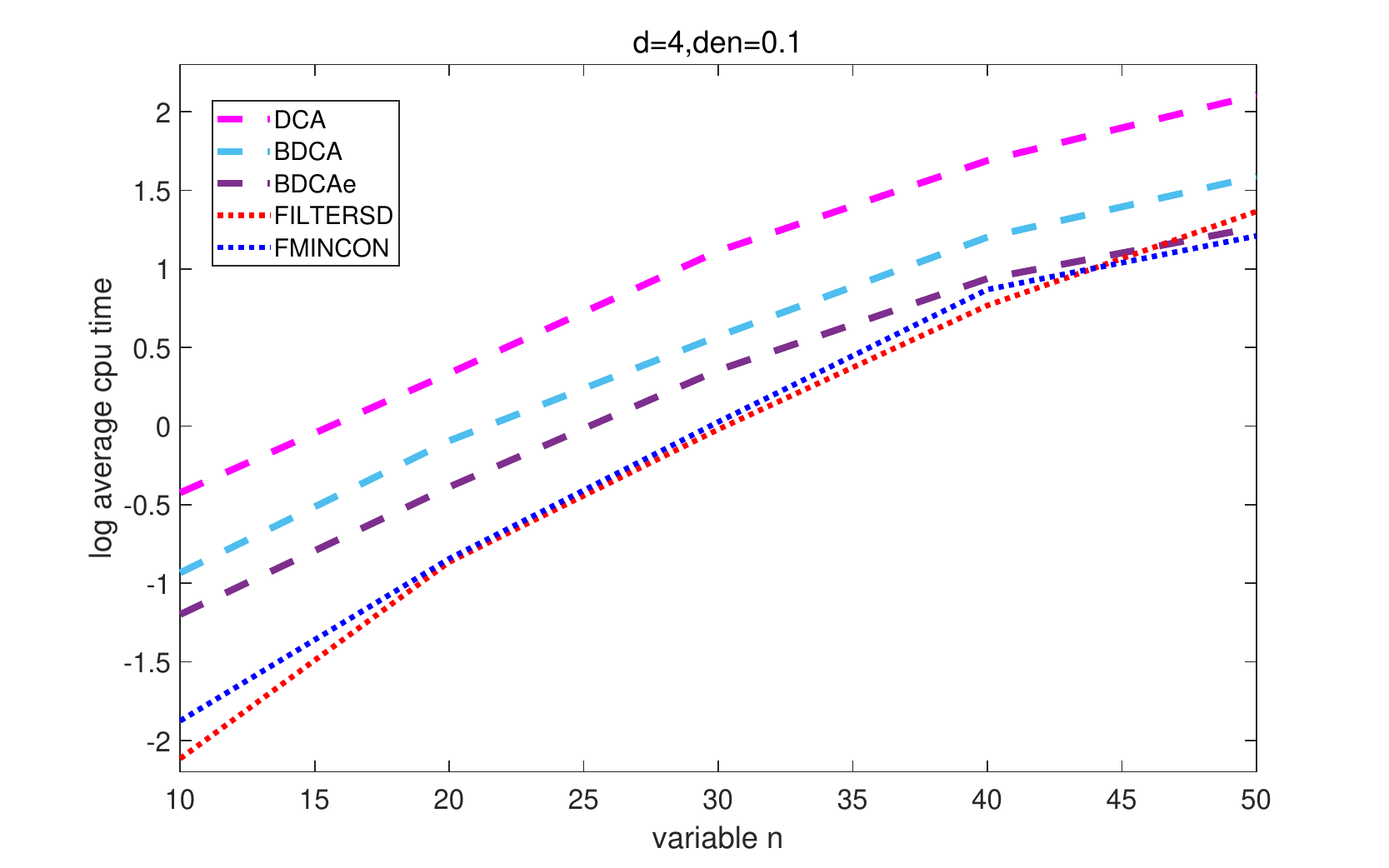}
 	\label{fig:d401}}
 	\subfigure{
 	\includegraphics[width=0.48\linewidth]{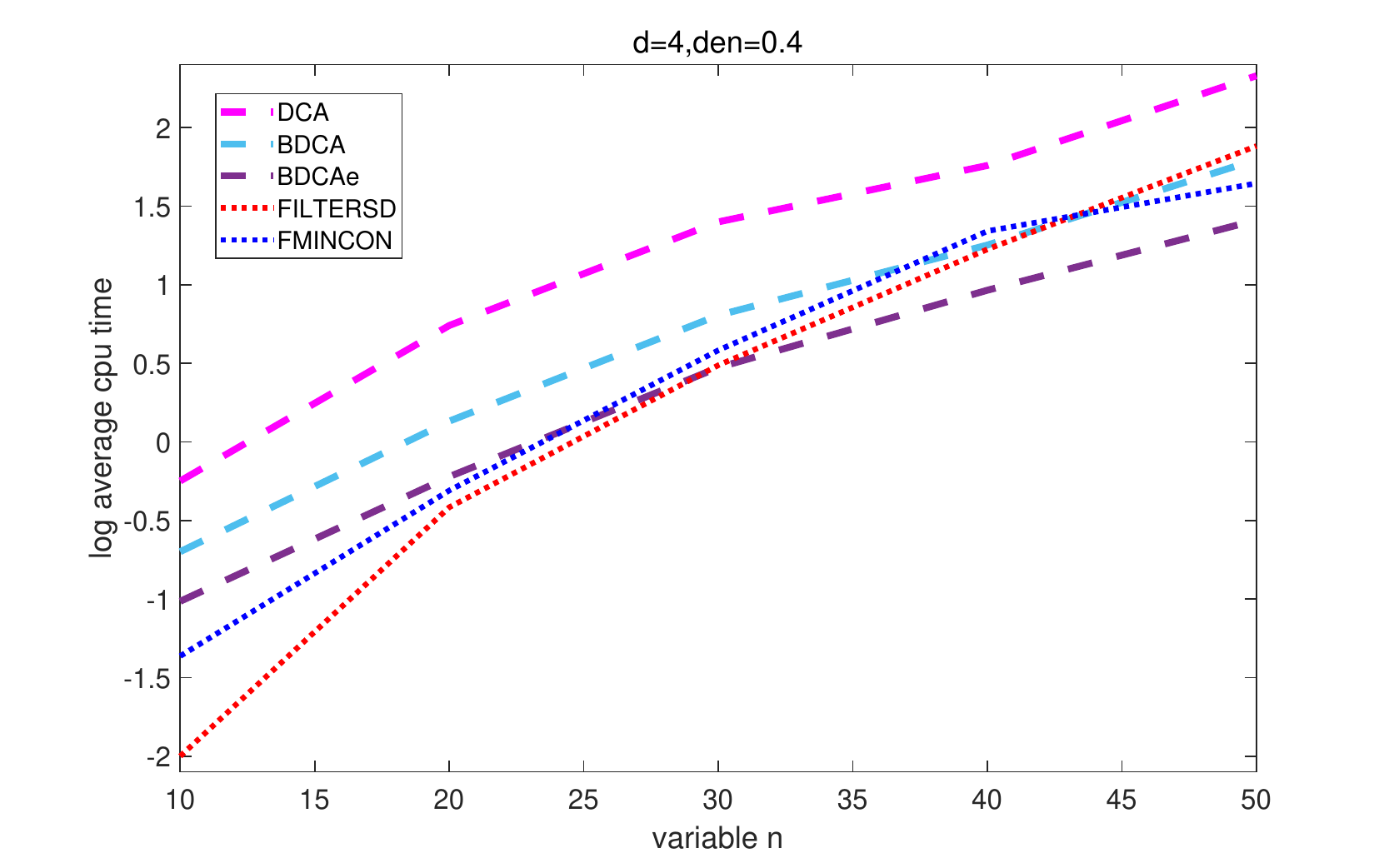}\label{fig:d404}}
 	\subfigure{
 	\includegraphics[width=0.48\linewidth]{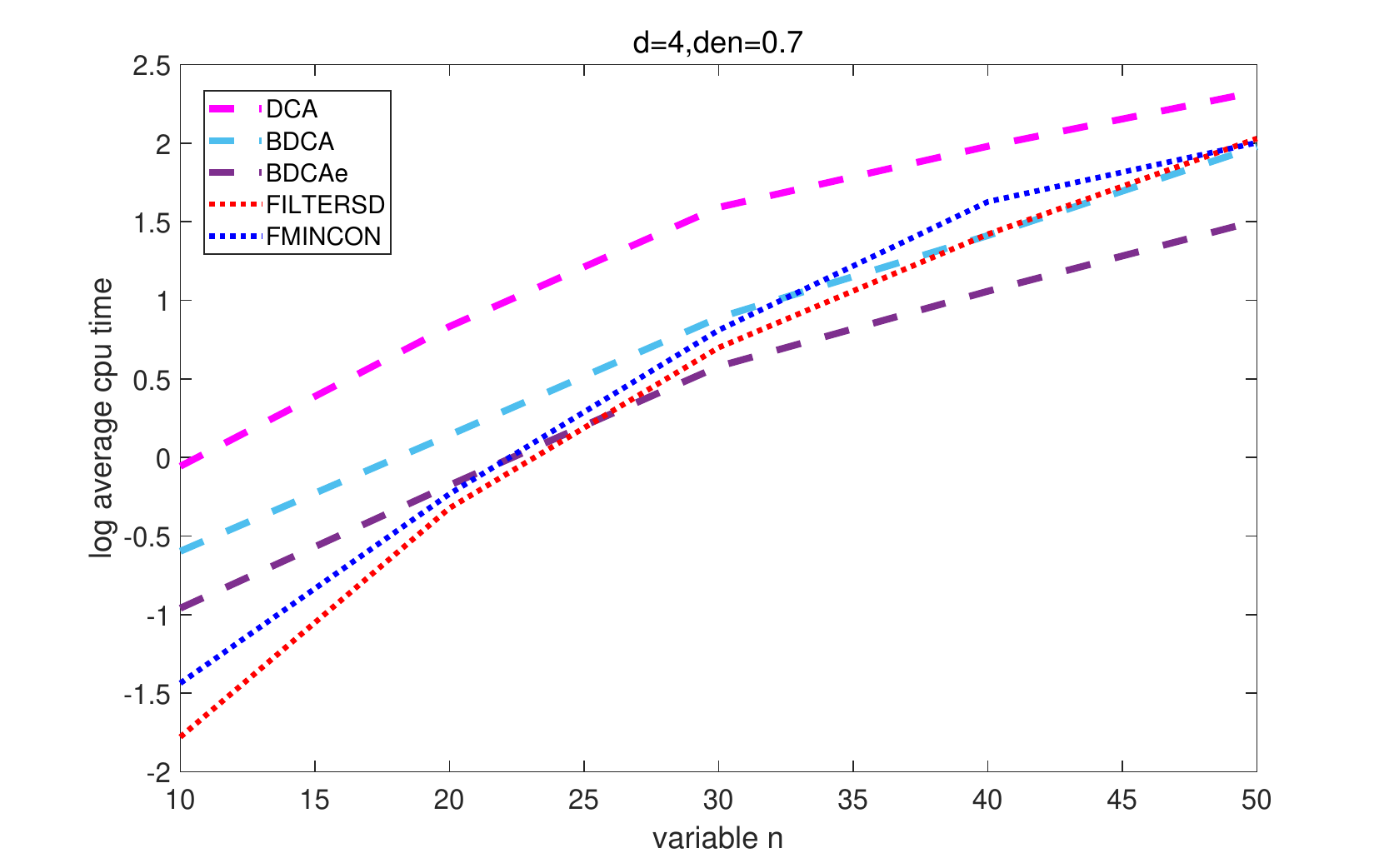}
 	\label{fig:d407}}
 	\subfigure{
 	\includegraphics[width=0.48\linewidth]{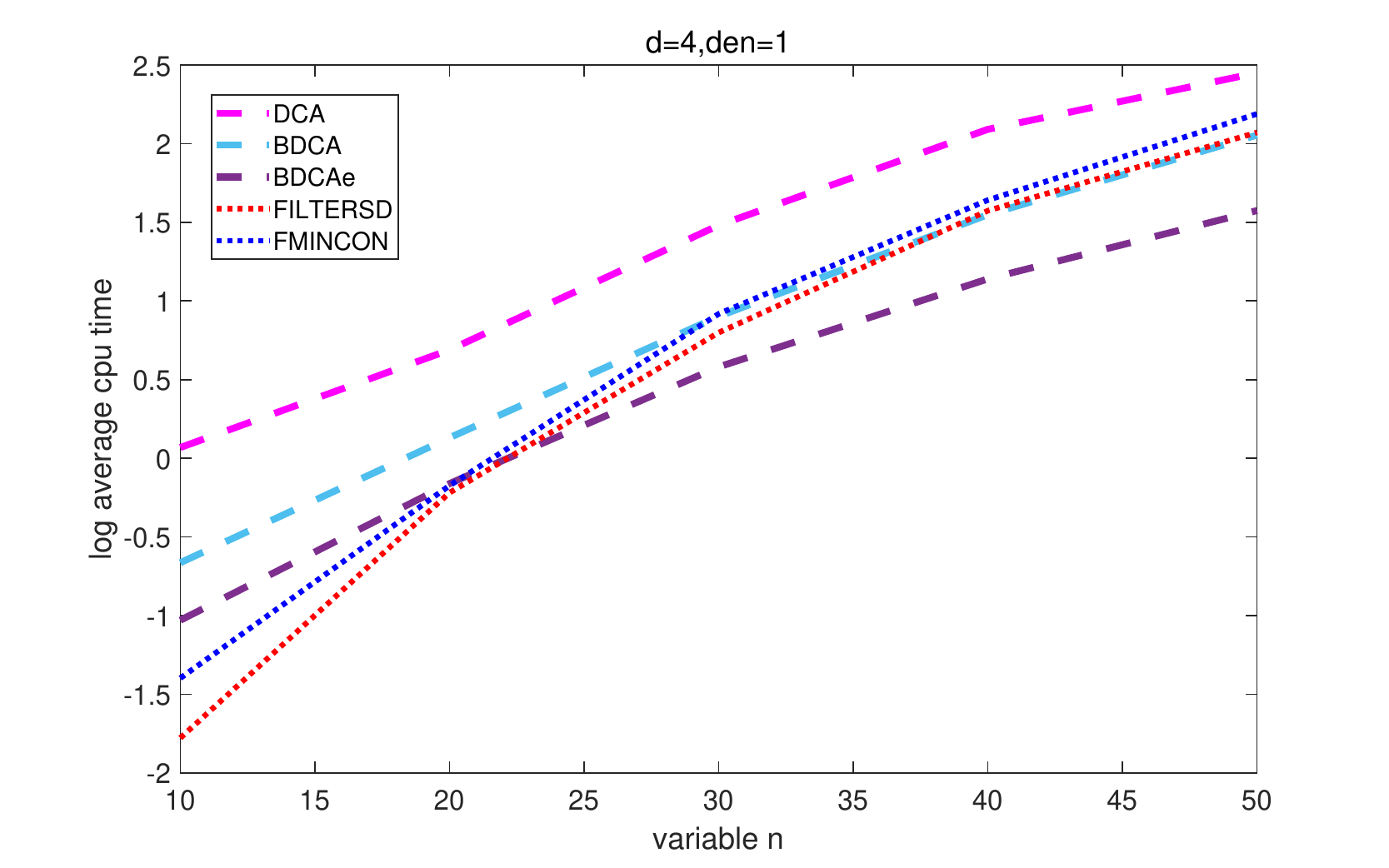}\label{fig:d41}}
 	\caption{Comparison of DCA, BDCA and BDCA$_\text{e}$ applied to HD-PSDC decomposition with $\rho=1$, as well as \texttt{FILTERSD} and \texttt{FMINCON} for solving \eqref{eq:Box} with $n\in\{10,20,30,40,50\},d=4, den\in\{0.1,0.4,0.7,1\}$.}
 	\label{fig:d-4}
\end{figure} 
Figures \ref{fig:d-3} and \ref{fig:d-4} show the performance of the five compared methods with $d=3$ and $4$ for all $n$ and $den$. One can see that, our BDCA$_\text{e}$ always outperforms DCA and BDCA in terms of the log average CPU time. In particular, for even $d$ as $d=4$ and for each $den$, BDCA$_\text{e}$ improves with respect to \texttt{FILTERSD} and \texttt{FMINCON} when $n$ increases, and for $den\geq 40\%$, BDCA$_\text{e}$ always outperforms \texttt{FILTERSD} and \texttt{FMINCON} when $n\geq 30$. However, when $d$ is odd (e.g., $d=3$), regardless of the values of $n$ and $den$, \texttt{FILTERSD} and \texttt{FMINCON} always outperform the others in terms of the log average CPU time. This is likely due to that the proposed power-sum DC decompositions transform an odd-degree polynomial into an even-degree polynomial, potentially increasing the overall computational time when solving \eqref{eq:Box}. Furthermore, we observe that BDCA$_\text{e}$ showcases reductions in log average CPU time compared to both \texttt{FILTERSD} and \texttt{FMINCON} as $d$ increases. This may be attributed to the fact that the computational cost required in \texttt{FILTERSD} and \texttt{FMINCON} escalates more rapidly than the expense incurred in solving a sequence of convex subproblems by leveraging the power-sum structure as the degree $d$ increases.  Additionally, it's noteworthy that BDCA$_\text{e}$ achieves objective values nearly identical to those of other methods across all tested scenarios.

Table \ref{tab:t_tablebox}-\ref{tab:hd_tablebox} summarizes the numerical results of BDCA$_\text{e}$, DCA, BDCA (with T-PSDC and HD-PSDC decompositions, respectively), \texttt{FILTERSD} and \texttt{FMINCON} with $n\in\{10,20,30,40,50\}$, $d=4$ and $den\in[0.5,1]$ for \eqref{eq:Box}. Here, we generate for each $n$ four instances with a random density in $[0.5,1]$. We observe that DCA, BDCA and BDCA$_\text{e}$ for \eqref{eq:HD-PSDC} consistently outperform the corresponding methods for \eqref{eq:T-PSDC}. Specifically, BDCA$\text{e}$ exhibits nearly identical average CPU time for both \eqref{eq:T-PSDC} and \eqref{eq:HD-PSDC}, and remarkably, they are the fastest methods (at least twice faster than the other methods in terms of the average CPU time). Furthermore, in about half of the examined scenarios, BDCA$_\text{e}$ for both PSDC decompositions achieves the best objective values. Fig. \ref{fig:log_cpu_time_box} depicts the average CPU time on a logarithmic scale as a function of $n$. Specifically, it shows the average CPU time over the four generated instances for each variable count, as presented in Table \ref{tab:t_tablebox}-\ref{tab:hd_tablebox}. We can see that BDCA$_\text{e}$ for both two PSDC decompositions gives the best performance when $n>25$, which confirms again that BDCA$_\text{e}$ performs better than the others for large-scale cases.

\begin{sidewaystable}[htbp]
 \centering
\resizebox{0.9\linewidth}{!}{
 \begin{tabular}{c|c|c|*{3}{c}|*{3}{c}|*{3}{c}}
 \toprule
 \multirow{3}*{}&\multirow{3}*{$n$}&\multirow{3}*{$den$}&\multicolumn{9}{c}{Algorithms for solving \eqref{eq:T-PSDC}}\\
 \cline{4-12}&&&\multicolumn{3}{c|}{DCA}&\multicolumn{3}{c|}{BDCA}&\multicolumn{3}{c}{BDCA$_\text{e}$}\\
\cline{4-12}&&&iter&time&obj&iter&time&obj&iter&time&obj\\
\midrule
1&10&0.98&275&1.88&\textbf{-2.622e+01}&31&0.35&\textbf{-2.622e+01}&17&0.14&\textbf{-2.622e+01}\\
2&10&0.96&154&1.19&\textbf{-3.097e+01}&33&0.45&\textbf{-3.097e+01}&16&0.17&\textbf{-3.097e+01}\\
3&10&0.59&57&0.83&\textbf{-3.983e+01}&19&0.38&\textbf{-3.983e+01}&11&0.22&\textbf{-3.983e+01}\\
4&10&0.77&146&1.27&-2.494e+01&24&0.34&-2.494e+01&11&0.16&\textbf{-2.495e+01}\\
5&20&0.92&454&6.43&-1.262e+02&71&1.94&-1.265e+02&36&0.84&-1.265e+02\\
6&20&0.57&382&5.85&-1.107e+02&45&1.46&-1.107e+02&21&0.67&-1.114e+02\\
7&20&0.79&410&11.50&\textbf{-1.625e+02}&59&2.14&\textbf{-1.625e+02}&23&0.88&\textbf{-1.625e+02}\\
8&20&0.62&185&6.02&\textbf{-1.446e+02}&42&1.69&\textbf{-1.446e+02}&21&0.87&\textbf{-1.446e+02}\\
9&30&0.52&195&23.45&-3.658e+02&65&7.93&-3.660e+02&31&3.94&-3.660e+02\\
10&30&0.94&290&37.15&\textbf{-4.805e+02}&78&10.93&\textbf{-4.805e+02}&32&5.00&\textbf{-4.805e+02}\\
11&30&0.99&209&30.30&-4.446e+02&61&10.13&-4.465e+02&32&5.22&-4.465e+02\\
12&30&0.70&335&36.11&-3.578e+02&67&9.44&-3.578e+02&31&4.54&-3.578e+02\\
13&40&0.70&457&116.89&\textbf{-7.695e+02}&90&31.65&\textbf{-7.695e+02}&70&14.69&\textbf{-7.695e+02}\\
14&40&0.55&373&96.48&-6.015e+02&106&27.87&-6.106e+02&41&11.98&-6.138e+02\\
15&40&0.53&598&91.58&-5.770e+02&107&23.06&-5.871e+02&49&11.65&-5.871e+02\\
16&40&0.52&540&111.63&\textbf{-7.274e+02}&123&39.82&\textbf{-7.274e+02}&42&11.56&\textbf{-7.274e+02}\\
17&50&0.52&635&264.68&-1.018e+03&116&63.97&-1.019e+03&62&31.72&-1.292e+03\\
18&50&0.58&673&240.01&-9.967e+02&139&113.73&-1.108e+03&57&33.31&-1.108e+03\\
19&50&0.73&482&235.82&-1.337e+03&107&81.11&-1.342e+03&48&36.73&\textbf{-1.349e+03}\\
20&50&0.66&655&259.53&-1.154e+03&120&145.07&-1.154e+03&43&33.57&-1.154e+03\\
\hline
\multicolumn{3}{c|}{average}&375&78.93&&75&28.67&&35&10.39\\
\bottomrule
\end{tabular}
}
\caption{Numerical results of DCA, BDCA and BDCA$_\text{e}$ (with T-PSDC decomposition and $\rho=1$) for \eqref{eq:Box}; Bold values for best results in objective function.}\label{tab:t_tablebox}
\end{sidewaystable}

\begin{sidewaystable}[htbp]
 \centering
\resizebox{\linewidth}{!}{
 \begin{tabular}{c|c|c|*{3}{c}|*{3}{c}|*{3}{c}|*{2}{c}|*{2}{c}}
 \toprule
  \multirow{3}*{}&\multirow{3}*{$n$}&\multirow{3}*{$den$}&\multicolumn{9}{c|}{Algorithms for solving \eqref{eq:HD-PSDC}}&\multicolumn{2}{c|}{\multirow{2}*{\texttt{FILTERSD}}}&\multicolumn{2}{c}{\multirow{2}*{\texttt{FMINCON}}}\\
 \cline{4-12}&&&\multicolumn{3}{c|}{DCA}&\multicolumn{3}{c|}{BDCA}&\multicolumn{3}{c|}{BDCA$_\text{e}$}&\multicolumn{2}{c|}{}&\multicolumn{2}{c}{}\\
\cline{4-16}&&&iter&time&obj&iter&time&obj&iter&time&obj&time&obj&time&obj\\
\midrule
1&10&0.98&281&0.97&\textbf{-2.622e+01}&37&0.22&\textbf{-2.622e+01}&17&0.08&\textbf{-2.622e+01}&0.02&\textbf{-2.622e+01}&0.07&\textbf{-2.622e+01}\\
2&10&0.96&146&0.74&\textbf{-3.097e+01}&29&0.23&\textbf{-3.097e+01}&15&0.10&\textbf{-3.097e+01}&0.02&\textbf{-3.097e+01}&0.04&-2.155e+01\\
3&10&0.59&54&0.64&\textbf{-3.983e+01}&18&0.21&\textbf{-3.983e+01}&11&0.14&\textbf{-3.983e+01}&0.01&\textbf{-3.983e+01}&0.02&\textbf{-3.983e+01}\\
4&10&0.77&135&0.96&\textbf{-2.495e+01}&26&0.25&\textbf{-2.495e+01}&11&0.10&\textbf{-2.495e+01}&0.02&\textbf{-2.495e+01}&0.07&\textbf{-2.495e+01}\\
5&20&0.92&445&5.19&-1.265e+02&70&1.45&-1.265e+02&36&0.65&-1.265e+02&0.94&-1.265e+02&1.20&\textbf{-1.386e+02}\\
6&20&0.57&430&4.33&-1.107e+02&45&1.10&-1.107e+02&21&0.52&-1.114e+02&0.40&-1.114e+02&0.55&\textbf{-1.227e+02}\\
7&20&0.79&425&7.49&\textbf{-1.625e+02}&57&1.49&\textbf{-1.625e+02}&23&0.60&\textbf{-1.625e+02}&0.57&\textbf{-1.625e+02}&0.94&\textbf{-1.625e+02}\\
8&20&0.62&172&4.03&\textbf{-1.446e+02}&42&1.19&\textbf{-1.446e+02}&21&0.60&\textbf{-1.446e+02}&0.38&\textbf{-1.446e+02}&0.61&-1.390e+02\\
9&30&0.52&198&20.93&-3.660e+02&63&7.35&-3.660e+02&31&3.73&-3.660e+02&4.18&-2.958e+02&6.57&\textbf{-4.037e+02}\\
10&30&0.94&286&25.74&\textbf{-4.805e+02}&66&7.93&\textbf{-4.805e+02}&30&3.77&\textbf{-4.805e+02}&5.48&\textbf{-4.805e+02}&6.93&\textbf{-4.805e+02}\\
11&30&0.99&199&24.11&-4.446e+02&63&7.68&-4.465e+02&32&3.76&-4.465e+02&5.38&-4.465e+02&5.86&\textbf{-4.933e+02}\\
12&30&0.70&334&29.42&-3.578e+02&66&7.96&-3.578e+02&31&3.73&-3.578e+02&4.64&-3.495e+02&3.43&\textbf{-3.634e+02}\\
13&40&0.70&453&91.54&\textbf{-7.695e+02}&88&24.70&\textbf{-7.695e+02}&62&10.90&\textbf{-7.695e+02}&30.04&-7.517e+02&14.46&\textbf{-7.695e+02}\\
14&40&0.55&332&77.69&-6.014e+02&92&23.84&-6.106e+02&38&10.26&-6.138e+02&21.83&\textbf{-6.658e+02}&21.58&-6.094e+02\\
15&40&0.53&575&76.55&-5.769e+02&119&24.87&-5.871e+02&52&10.10&-5.871e+02&29.90&-5.871e+02&31.06&\textbf{-7.517e+02}\\
16&40&0.52&514&94.29&\textbf{-7.274e+02}&125&44.22&\textbf{-7.274e+02}&44&9.32&\textbf{-7.274e+02}&20.35&\textbf{-7.274e+02}&23.93&-6.843e+02\\
17&50&0.52&612&218.13&-1.014e+03&129&71.44&-1.019e+03&61&25.03&\textbf{-1.295e+03}&86.10&-9.023e+02&42.68&-1.145e+03\\
18&50&0.58&386&248.07&-1.108e+03&96&88.39&-1.108e+03&52&33.27&-1.108e+03&110.35&-1.108e+03&136.60&\textbf{-1.150e+03}\\
19&50&0.73&489&206.95&-1.337e+03&114&69.15&\textbf{-1.349e+03}&57&34.30&\textbf{-1.349e+03}&118.62&-1.203e+03&80.83&-1.276e+03\\
20&50&0.66&494&206.30&-1.154e+03&105&74.94&-1.154e+03&55&29.47&-1.154e+03&111.47&-1.223e+03&83.72&\textbf{-1.275e+03}\\
\hline
\multicolumn{3}{c|}{average}&348&67.20&&73&22.93&&35&9.02&&27.53&&23.06&\\
\bottomrule
\end{tabular}
}
\caption{Numerical results of DCA, BDCA and BDCA$_\text{e}$ (with HD-PSDC decomposition and $\rho=1$), as well as \texttt{FILTERSD} and \texttt{FMINCON} for \eqref{eq:Box}; Bold values for best results in objective function.}\label{tab:hd_tablebox}
\end{sidewaystable} 
\begin{figure}[ht!]
	\centering
	\includegraphics[width=10cm,height=6cm]{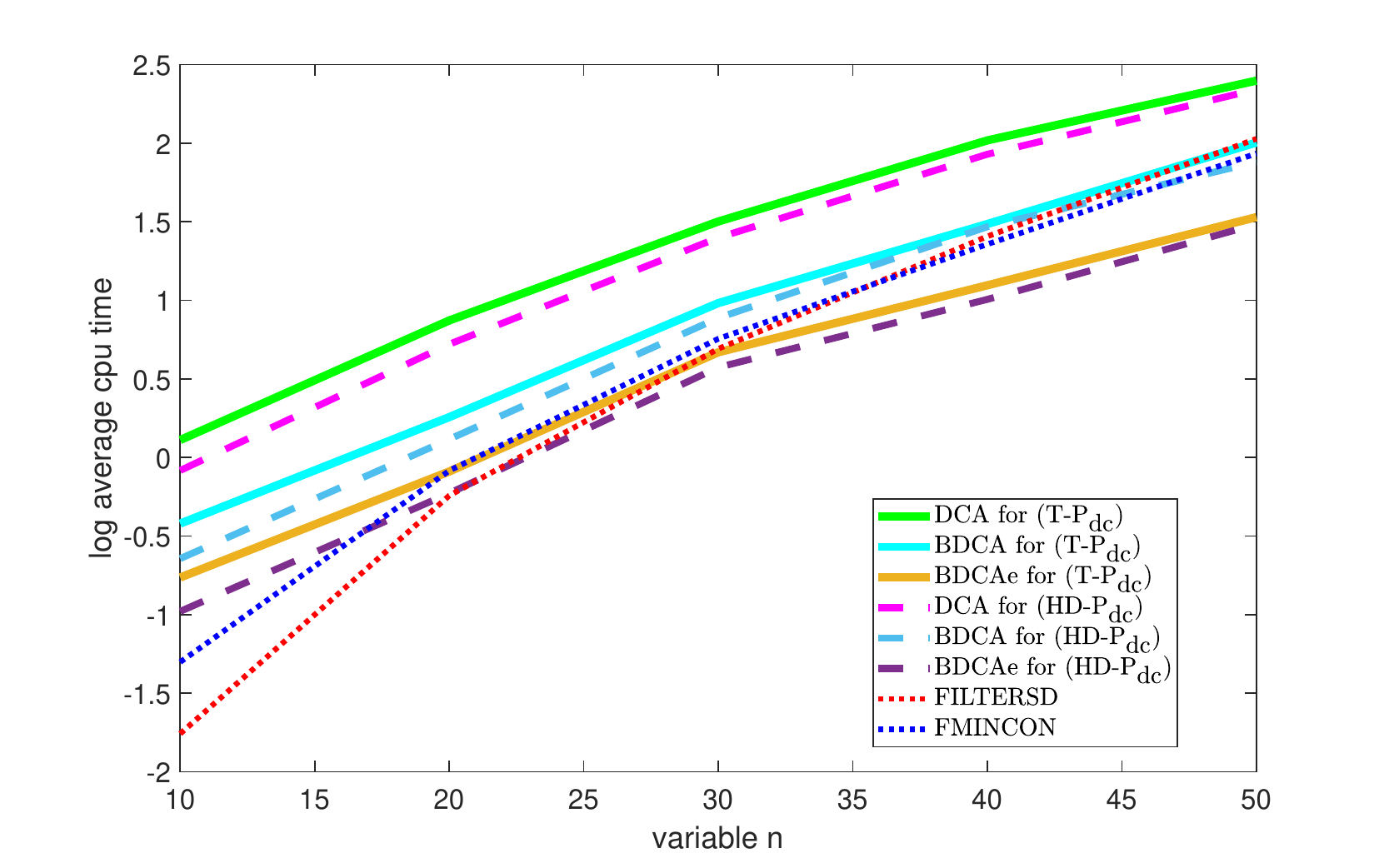}
	\caption{Log average CPU time along $n$ for results of DCA, BDCA and BDCA$_\text{e}$ for \eqref{eq:T-PSDC} and \eqref{eq:HD-PSDC} as well as \texttt{FILTERSD} and \texttt{FMINCON} reported in Table \ref{tab:t_tablebox} and Table \ref{tab:hd_tablebox}.}
	\label{fig:log_cpu_time_box}
\end{figure}

\section{Conclusions and future work}\label{sec:conclusion}
In this paper, we presented a novel DC-SOS decomposition technique based on the power-sum representation for polynomials. Then we proposed an accelerated DCA (BDCA$_\text{e}$) combining DCA with an exact line search along the DC descent direction for solving linearly constrained polynomial optimization problems. We proved that any limit point of the sequence $\{\x^k\}$ generated by BDCA$_\text{e}$ is a critical point of \eqref{eq:POPL} and established the convergence rate of the sequence $\{f(\x^k)\}$ under the Kurdyka-{\L}ojasiewicz assumption. We applied the FDPG method by exploiting the power-sum-DC structure for efficiently solving the convex subproblems required in BDCA$_{e}$. The numerical experiments on the higher-order moment portfolio optimization model \eqref{eq:MVSK} and the box-constrained polynomial optimization problem \eqref{eq:Box} have demonstrated that BDCA$_\text{e}$ is often comparable to \texttt{FILTERSD} and \texttt{FMINCON} in terms of the CPU time and objective values, and outperforms DCA, UDCA, BDCA and UBDCA on the tested dataset. Moreover, BDCA$_\text{e}$ shows advantages in solving large-scale dense polynomial optimization problems. It's worth noting that there are three major benefits for our proposed power-sum DC decompositions. First, they provide good DC decompositions for polynomials which can be verified in numerical results that DCA with power-sum DC decompositions requires fewer iterations than DCA with the classical projective DC decomposition. Second, the convex subproblem in BDCA$_\text{e}$ can be solved efficiently in parallel by FDGP as discussed in Section \ref{sec:FDPG}, which cannot be the case without the particular power-sum structure. Third, the exact line search in Section \ref{sec:DCAe} amounts to finding the roots of a univariate polynomial whose coefficients are computed explicitly based on our power-sum decompositions, leading to efficient exact line search in some real-world applications.

Several pivotal questions warrant further exploration. First, the generation of power-sum DC decompositions with minimal square terms arises as a notable inquiry due to its crucial influence on the efficiency of DCA and its variants.
Second, the pursuit of an adaptive methodology for updating the strongly convex parameter  $\rho$ remains essential to ensure robust numerical performance in BDCA$_{e}$. It's worth noting the intrinsic trade-off regarding the choice of modulus \(\rho\) concerning the overall performance of BDCA\(_\text{e}\). A diminutive modulus \(\rho\) engenders a superior DC decomposition, thereby reducing the requisite iterations for DCA\(_\text{e}\); however, as articulated by Lemma \ref{lem:convergence-rate-xl}, a magnified \(\rho\) accelerates the convergence of FDPG in solving the convex subproblem. Therefore, the quest for an optimal \(\rho\) to assure peak performance is imperative.

\begin{acknowledgements}
We extend our sincere gratitude to the anonymous referees for their insightful suggestions which significantly improved the quality of this manuscript.
\end{acknowledgements}

\begin{myblock}{\fundingname}
This work was supported by the Natural Science Foundation of China (Grant No: 11601327).    
\end{myblock}

\begin{myblock}{Data Availability}
All data generated and analysed during this study are available from the corresponding author on reasonable request. 
\end{myblock}

\begin{myblock}{Conflict of Interest}
The authors declare that they have no conflict of interest. 
\end{myblock}

\bibliographystyle{spmpsci}     
\bibliography{mybibfile}   

\begin{thebibliography}{10}
\providecommand{\url}[1]{{#1}}
\providecommand{\urlprefix}{URL }
\expandafter\ifx\csname urlstyle\endcsname\relax
  \providecommand{\doi}[1]{DOI~\discretionary{}{}{}#1}\else
  \providecommand{\doi}{DOI~\discretionary{}{}{}\begingroup
  \urlstyle{rm}\Url}\fi

\bibitem{ahmadi2018dc}
Ahmadi, A.A., Hall, G.: {DC} decomposition of nonconvex polynomials with
  algebraic techniques.
\newblock Mathematical Programming \textbf{169}(1), 69--94 (2018)

\bibitem{ahmadi2012convex}
Ahmadi, A.A., Parrilo, P.A.: A convex polynomial that is not sos-convex.
\newblock Mathematical Programming \textbf{135}(1), 275--292 (2012)

\bibitem{ahmadi2013complete}
Ahmadi, A.A., Parrilo, P.A.: A complete characterization of the gap between
  convexity and sos-convexity.
\newblock SIAM Journal on Optimization \textbf{23}(2), 811--833 (2013)

\bibitem{anjos2011handbook}
Anjos, M.F., Lasserre, J.B.: Handbook on {S}emidefinite, {C}onic and
  {P}olynomial {O}ptimization, vol. 166.
\newblock Springer Science \& Business Media (2011)

\bibitem{aps2019mosek}
ApS, M.: Mosek {O}ptimization {T}oolbox for {MATLAB}.
\newblock User’s Guide and Reference Manual, Version \textbf{4} (2019)

\bibitem{artacho2018accelerating}
Artacho, F.J.A., Fleming, R.M., Vuong, P.T.: Accelerating the {DC} algorithm
  for smooth functions.
\newblock Mathematical Programming \textbf{169}(1), 95--118 (2018)

\bibitem{aragon2018boosted}
Artacho, F.J.A., Vuong, P.T.: The boosted {DC} algorithm for nonsmooth
  functions.
\newblock SIAM Journal on Optimization \textbf{30}(1), 980--1006 (2020)

\bibitem{ATHAYDE2003243}
Athayde, G.M.D., Fl\^ores, R.G.: Incorporating skewness and kurtosis in
  portfolio optimization: A multidimensional efficient set.
\newblock In: Advances in Portfolio Construction and Implementation,
  Quantitative Finance, pp. 243--257. Elsevier (2003)

\bibitem{bakonyi1995euclidian}
Bakonyi, M., Johnson, C.R.: The euclidian distance matrix completion problem.
\newblock SIAM Journal on Matrix Analysis and Applications \textbf{16}(2),
  646--654 (1995)

\bibitem{beck2017first}
Beck, A.: First-order {M}ethods in {O}ptimization.
\newblock SIAM (2017)

\bibitem{beck2009fast}
Beck, A., Teboulle, M.: A fast iterative shrinkage-thresholding algorithm for
  linear inverse problems.
\newblock SIAM journal on imaging sciences \textbf{2}(1), 183--202 (2009)

\bibitem{beck2014fast}
Beck, A., Teboulle, M.: A fast dual proximal gradient algorithm for convex
  minimization and applications.
\newblock Operations Research Letters \textbf{42}(1), 1--6 (2014)

\bibitem{biosca2001representation}
Biosca, A.F.: Representation of a polynomial function as a difference of convex
  polynomials, with an application.
\newblock In: Generalized Convexity and Generalized Monotonicity, Lecture Notes
  in Economics and Mathematical Systems, pp. 189--207. Springer (2001)

\bibitem{bolte2007lojasiewicz}
Bolte, J., Daniilidis, A., Lewis, A.: The {{\L}}ojasiewicz inequality for
  nonsmooth subanalytic functions with applications to subgradient dynamical
  systems.
\newblock SIAM Journal on Optimization \textbf{17}(4), 1205--1223 (2007)

\bibitem{bomze2004undominated}
Bomze, I.M., Locatelli, M.: Undominated dc decompositions of quadratic
  functions and applications to branch-and-bound approaches.
\newblock Computational Optimization and Applications \textbf{28}(2), 227--245
  (2004)

\bibitem{boyd1994linear}
Boyd, S., El~Ghaoui, L., Feron, E., Balakrishnan, V.: {L}inear {M}atrix
  {I}nequalities in {S}ystem and {C}ontrol {T}heory.
\newblock SIAM (1994)

\bibitem{bras2016copositivity}
Br{\'a}s, C., Eichfelder, G., J{\'u}dice, J.: Copositivity tests based on the
  linear complementarity problem.
\newblock Computational Optimization and Applications \textbf{63}(2), 461--493
  (2016)

\bibitem{comon1996decomposition}
Comon, P., Mourrain, B.: Decomposition of quantics in sums of powers of linear
  forms.
\newblock Signal Processing \textbf{53}(2-3), 93--107 (1996)

\bibitem{de2017improved}
De~Klerk, E., Hess, R., Laurent, M.: Improved convergence rates for
  {L}asserre-type hierarchies of upper bounds for box-constrained polynomial
  optimization.
\newblock SIAM Journal on Optimization \textbf{27}(1), 347--367 (2017)

\bibitem{dur2013testing}
D{\"u}r, M., Hiriart-Urruty, J.B.: Testing copositivity with the help of
  difference-of-convex optimization.
\newblock Mathematical Programming \textbf{140}(1), 31--43 (2013)

\bibitem{filtersdFletcher}
Fletcher, R.: Filtersd–a {L}ibrary for {N}onlinear {O}ptimization {W}ritten
  in {F}ortran (2015).
\newblock \urlprefix\url{https://projects.coin-or.org/filterSD}

\bibitem{frankel2015splitting}
Frankel, P., Garrigos, G., Peypouquet, J.: Splitting methods with variable
  metric for {K}urdyka--{{\L}}ojasiewicz functions and general convergence
  rates.
\newblock Journal of Optimization Theory and Applications \textbf{165}(3),
  874--900 (2015)

\bibitem{gotoh2018dc}
Gotoh, J., Takeda, A., Tono, K.: D{C} formulations and algorithms for sparse
  optimization problems.
\newblock Mathematical Programming \textbf{169}(1), 141--176 (2018)

\bibitem{hartman1959functions}
Hartman, P.: On functions representable as a difference of convex functions.
\newblock Pacific Journal of Mathematics \textbf{9}(3), 707--713 (1959)

\bibitem{hiriart1985generalized}
Hiriart-Urruty, J.B.: Generalized differentiability, duality and optimization
  for problems dealing with differences of convex functions.
\newblock In: Convexity and duality in optimization, Lecture Notes in Economics
  and Mathematical Systems, pp. 37--70. Springer (1985)

\bibitem{kurdyka1994wf}
Kurdyka, K., Parusinski, A.: wf-stratification of subanalytic functions and the
  {{\L}}ojasiewicz inequality.
\newblock Comptes rendus de l'Acad{\'e}mie des sciences. S{\'e}rie 1,
  Math{\'e}matique \textbf{318}(2), 129--133 (1994)

\bibitem{hoai2000efficient}
Le~Thi, H.A.: An efficient algorithm for globally minimizing a quadratic
  function under convex quadratic constraints.
\newblock Mathematical programming \textbf{87}(3), 401--426 (2000)

\bibitem{le2012dc}
Le~Thi, H.A., Moeini, M., Pham, D.T., Judice, J.: A {DC} programming approach
  for solving the symmetric eigenvalue complementarity problem.
\newblock Computational Optimization and Applications \textbf{51}(3),
  1097--1117 (2012)

\bibitem{le1997solving}
Le~Thi, H.A., Pham, D.T.: Solving a class of linearly constrained indefinite
  quadratic problems by {DC} algorithms.
\newblock Journal of global optimization \textbf{11}(3), 253--285 (1997)

\bibitem{tao2005dc}
Le~Thi, H.A., Pham, D.T.: The {DC} (difference of convex functions) programming
  and {DCA} revisited with {DC} models of real world nonconvex optimization
  problems.
\newblock Annals of operations research \textbf{133}(1), 23--46 (2005)

\bibitem{le2018dc}
Le~Thi, H.A., Pham, D.T.: {DC} programming and {DCA}: thirty years of
  developments.
\newblock Mathematical Programming \textbf{169}(1), 5--68 (2018)

\bibitem{lojasiewicz1993geometrie}
{\L}ojasiewicz, S.: Sur {L}a {G}{\'e}om{\'e}trie {S}emi et {S}ous-analytique.
\newblock In: Annales de l'institut Fourier, pp. 1575--1595 (1993)

\bibitem{Maringer2009GlobalOO}
Maringer, D., Parpas, P.: Global optimization of higher order moments in
  portfolio selection.
\newblock Journal of Global optimization \textbf{43}(2), 219--230 (2009)

\bibitem{megretski2013spot}
Megretski, A.: {SPOT}: {S}ystems {P}olynomial {O}ptimization {T}ools (2013)

\bibitem{nesterov1983method}
Nesterov, Y.E.: A method for solving the convex programming problem with
  convergence rate {$O(1/k^{2})$}.
\newblock Dokl. Akad. Nauk SSSR \textbf{269}(3), 543--547 (1983)

\bibitem{niu2018difference}
Niu, Y.S.: On difference-of-sos and difference-of-convex-sos decompositions for
  polynomials.
\newblock arXiv :1803.09900  (2018)

\bibitem{niu2019discrete}
Niu, Y.S., Glowinski, R.: Discrete dynamical system approaches for boolean
  polynomial optimization.
\newblock Journal of Scientific Computing \textbf{92}(2), 1--39 (2022)

\bibitem{niu2015solving}
Niu, Y.S., J{\'u}dice, J., Le~Thi, H.A., Pham, D.T.: Solving the quadratic
  eigenvalue complementarity problem by {DC} programming.
\newblock In: Modelling, Computation and Optimization in Information Systems
  and Management Sciences, Advances in Intelligent Systems and Computing, pp.
  203--214. Springer (2015)

\bibitem{niu2019improved}
Niu, Y.S., J{\'u}dice, J., Le~Thi, H.A., Pham, D.T.: Improved dc programming
  approaches for solving the quadratic eigenvalue complementarity problem.
\newblock Applied Mathematics and Computation \textbf{353}, 95--113 (2019)

\bibitem{niu2014dc}
Niu, Y.S., Pham, D.T.: D{C} programming approaches for {BMI} and {QMI}
  feasibility problems.
\newblock In: Advanced Computational Methods for Knowledge Engineering,
  Advances in Intelligent Systems and Computing, pp. 37--63. Springer (2014)

\bibitem{niu2013efficient}
Niu, Y.S., Pham, D.T., Le~Thi, H.A., J{\'u}dice, J.: Efficient {DC} programming
  approaches for the asymmetric eigenvalue complementarity problem.
\newblock Optimization Methods and Software \textbf{28}(4), 812--829 (2013)

\bibitem{niu2019higher}
Niu, Y.S., Wang, Y.J., Le~Thi, H.A., Pham, D.T.: Higher-order moment portfolio
  optimization via an accelerated difference-of-convex programming approach and
  sums-of-squares.
\newblock arXiv :1906.01509  (2019)

\bibitem{niu2023power}
Niu, Y.S., Zhang, H.: Power-product matrix: nonsingularity, sparsity and
  determinant.
\newblock Linear and Multilinear Algebra \textbf{0}(0), 1--18 (2023)

\bibitem{de2019inertial}
de~Oliveira, W., Tcheou, M.P.: An inertial algorithm for dc programming.
\newblock Set-Valued and Variational Analysis \textbf{27}(4), 895--919 (2019)

\bibitem{parlett1998symmetric}
Parlett, B.N.: The {S}ymmetric {E}igenvalue {P}roblem.
\newblock SIAM (1998)

\bibitem{1997Convex}
Pham, D.T., Le~Thi, H.A.: Convex analysis approach to {DC} programming: theory,
  algorithms and applications.
\newblock Acta mathematica vietnamica \textbf{22}(1), 289--355 (1997)

\bibitem{tao1998dc}
Pham, D.T., Le~Thi, H.A.: A {DC} optimization algorithm for solving the
  trust-region subproblem.
\newblock SIAM Journal on Optimization \textbf{8}(2), 476--505 (1998)

\bibitem{dinh2016dc}
Pham, D.T., Le~Thi, H.A., Pham, V.N., Niu, Y.S.: D{C} programming approaches
  for discrete portfolio optimization under concave transaction costs.
\newblock Optimization letters \textbf{10}(2), 261--282 (2016)

\bibitem{Niu2011}
Pham, D.T., Niu, Y.S.: An efficient {DC} programming approach for portfolio
  decision with higher moments.
\newblock Computational Optimization and Applications \textbf{50}(3), 525--554
  (2011)

\bibitem{1988Duality}
Pham, D.T., Souad, E.B.: Duality in {DC} (difference of convex functions)
  optimization. subgradient methods.
\newblock In: Trends in Mathematical Optimization: 4th French-German Conference
  on Optimization, International Series of Numerical Mathematics, pp. 277--293.
  Springer (1988)

\bibitem{nhat2018accelerated}
Phan, D.N., Le, H.M., Le~Thi, H.A.: Accelerated {D}ifference of {C}onvex
  functions {A}lgorithm and its {A}pplication to {S}parse {B}inary {L}ogistic
  {R}egression.
\newblock In: Proceedings of the 27th International Joint Conference on
  Artificial Intelligence and the 23rd European Conference on Artificial
  Intelligence, pp. 1369--1375 (2018)

\bibitem{polyak1964some}
Polyak, B.T.: Some methods of speeding up the convergence of iteration methods.
\newblock Ussr computational mathematics and mathematical physics
  \textbf{4}(5), 1--17 (1964)

\bibitem{reznick1992sum}
Reznick, B.A.: Sum of {E}ven {P}owers of {R}eal {L}inear {F}orms.
\newblock American Mathematical Society (1992)

\bibitem{song2016properties}
Song, Y., Yu, G.: Properties of solution set of tensor complementarity problem.
\newblock Journal of Optimization Theory and Applications \textbf{170}(1),
  85--96 (2016)

\bibitem{wen2018proximal}
Wen, B., Chen, X., Pong, T.K.: A proximal difference-of-convex algorithm with
  extrapolation.
\newblock Computational optimization and applications \textbf{69}(2), 297--324
  (2018)

\bibitem{you2022refined}
You, Y., Niu, Y.S.: A refined inertial dc algorithm for dc programming.
\newblock Optimization and Engineering \textbf{24}(1), 65--91 (2023)

\end{thebibliography}
\end{document}